\documentclass[hidelinks,12pt]{amsart}
\usepackage{tikz,hyperref,stmaryrd,a4wide,amssymb,enumerate}

\usepackage[utf8]{inputenc}     % Encoding of the file
\usepackage{mathtools}
\usepackage{amsfonts}
\usepackage{amsthm}
\usepackage{amssymb}
\usepackage{amsmath}
\usepackage{tikz}
\usetikzlibrary{arrows,automata}
\usepackage{cite}
\usepackage{amsmath}
\usepackage{hyperref}
\usepackage{mathtools}

\usepackage{xcolor}

\newcommand{\tcb}{\textcolor{blue}}
\newcommand{\tcr}{\textcolor{red}}

\newtheorem{theorem}{Theorem}

\newtheorem{lemma}[theorem]{Lemma}
\newtheorem{proposition}[theorem]{Proposition}
\newtheorem{example}[theorem]{Example}
\newtheorem{corollary}[theorem]{Corollary}
\newtheorem{remark}[theorem]{Remark}

\newcommand{\rvline}{\hspace*{-\arraycolsep}\vline\hspace*{-\arraycolsep}}

\newcommand{\weight}[0]{\operatorname{weight}}
\newcommand{\trans}[1]{\mathchoice{\xrightarrow{#1}}{\xrightarrow{\smash{\lower1pt\hbox{$\scriptstyle #1$}}}}{\xrightarrow{#1}}{\xrightarrow{#1}}}
\newcommand{\tuple}[1]{\langle #1 \rangle}

\DeclareUnicodeCharacter{D7}{\times}            % ×
\DeclareUnicodeCharacter{393}{\Gamma}           % Γ
\DeclareUnicodeCharacter{394}{\Delta}           % Δ
\DeclareUnicodeCharacter{39B}{\Lambda}          % Λ
\DeclareUnicodeCharacter{3A6}{\Phi}             % Φ
\DeclareUnicodeCharacter{3A9}{\Omega}
\DeclareUnicodeCharacter{3B1}{\alpha}           % α
\DeclareUnicodeCharacter{3B2}{\beta}            % β
\DeclareUnicodeCharacter{3B3}{\gamma}           % γ
\DeclareUnicodeCharacter{3B4}{\delta}           % δ
\DeclareUnicodeCharacter{3B5}{\varepsilon}      % ε
\DeclareUnicodeCharacter{3B8}{\theta}
\DeclareUnicodeCharacter{3BB}{\lambda}          % λ
\DeclareUnicodeCharacter{3BC}{\mu}              % μ
\DeclareUnicodeCharacter{3BD}{\nu}              % ν
\DeclareUnicodeCharacter{3C0}{\pi}              % π 
\DeclareUnicodeCharacter{3C4}{\tau}             % τ
\DeclareUnicodeCharacter{3C6}{\varphi}          % φ
\DeclareUnicodeCharacter{3C7}{\chi}             % χ
\DeclareUnicodeCharacter{3D5}{\phi}             % ϕ
\DeclareUnicodeCharacter{2026}{\ldots}          % …
\DeclareUnicodeCharacter{2113}{\ell}            % ℓ
\DeclareUnicodeCharacter{2115}{\mathbb{N}}      % ℕ
\DeclareUnicodeCharacter{211D}{\mathbb{R}}      % ℝ
\DeclareUnicodeCharacter{2124}{\mathbb{Z}}      % ℤ
\DeclareUnicodeCharacter{212C}{\mathcal{B}}     % ℬ
\DeclareUnicodeCharacter{2131}{\mathcal{F}}     % ℱ
\DeclareUnicodeCharacter{2192}{\rightarrow}     % →
\DeclareUnicodeCharacter{21A6}{\mapsto}         % ↦
\DeclareUnicodeCharacter{2203}{\exists}         % ∃ 
\DeclareUnicodeCharacter{2208}{\in}             % ∈
\DeclareUnicodeCharacter{2209}{\notin}          % ∉
\DeclareUnicodeCharacter{220F}{\prod}           % ∏
\DeclareUnicodeCharacter{2211}{\sum}            % ∑
\DeclareUnicodeCharacter{2216}{\setminus}       % ∖
\DeclareUnicodeCharacter{2218}{\circ}           % ∘
\DeclareUnicodeCharacter{221E}{\infty}          % ∞
\DeclareUnicodeCharacter{2229}{\cap}            % ∩
\DeclareUnicodeCharacter{222A}{\cup}            % ∪
\DeclareUnicodeCharacter{2254}{\vcentcolon=}    % ≔ 
\DeclareUnicodeCharacter{2260}{\neq}            % ≠ 
\DeclareUnicodeCharacter{2261}{\equiv}          % ≡
\DeclareUnicodeCharacter{2282}{\subset}         % ⊂
\DeclareUnicodeCharacter{2295}{\oplus}          % ⊕
\DeclareUnicodeCharacter{2297}{\otimes}         % ⊗
\DeclareUnicodeCharacter{229E}{\boxplus}        % ⊞
\DeclareUnicodeCharacter{229F}{\boxminus}       % ⊟
\DeclareUnicodeCharacter{22C5}{\cdot}           % ⋅
\DeclareUnicodeCharacter{22EF}{\cdots}          % ⋯
\DeclareUnicodeCharacter{2308}{\lceil}          % ⌈
\DeclareUnicodeCharacter{2309}{\rceil}          % ⌉
\DeclareUnicodeCharacter{230A}{\lfloor}         % ⌊
\DeclareUnicodeCharacter{230B}{\rfloor}         % ⌋
\DeclareUnicodeCharacter{27E6}{\llbracket}      % ⟦
\DeclareUnicodeCharacter{27E7}{\rrbracket}      % ⟧
\DeclareUnicodeCharacter{27E8}{\langle}         % ⟨
\DeclareUnicodeCharacter{27E9}{\rangle}         % ⟩
\DeclareUnicodeCharacter{2A7D}{\leqslant}       % ⩽
\DeclareUnicodeCharacter{2A7E}{\geqslant}       % ⩾
\DeclareUnicodeCharacter{1D49C}{\mathcal{A}}    % \mathcal{A}
\DeclareUnicodeCharacter{1D49E}{\mathcal{C}}    % \mathcal{C}
\DeclareUnicodeCharacter{1D49F}{\mathcal{D}}    % \mathcal{D}
\DeclareUnicodeCharacter{1D539}{\mathbb{B}}     % \mathbb{B}
\DeclareUnicodeCharacter{1D53D}{\mathbb{F}}     % \mathbb{F}

\begin{document}
\title{Mahler equations for Zeckendorf numeration}
 
\thanks{The first, author was supported by the Agence Nationale de la
  Recherche through the project ``SymDynAr'' (ANR-23-CE40-0024-01).  The
  second author was supported by the EPSRC, grant number EP/V007459/2.}
\author[O.~Carton]{Olivier Carton}
\address{Institut universitaire de France et IRIF, Université Paris-Cité, France}
\author[R. Yassawi]{Reem Yassawi}
\address{School of Mathematical Sciences, Queen Mary University of London,
  United Kingdom}
\email{Olivier.Carton@irif.fr}
\email{r.yassawi@qmul.ac.uk}
\thanks{}
\date{\today}
\keywords{weighted automata; automata sequences; Zeckendorf numeration; Pisot numerations, Mahler equations}
\subjclass[2020]{11B85, 68Q45}
\maketitle
\begin{abstract}
  A Pisot numeration system $U=(u_n)_{n⩾ 0}$ for $ℕ$ is one where the
  sequence of positive integers $(u_n)_{n⩾ 0}$, with $u_n=1$, is generated
  by a recurrence whose polynomial is the minimal polynomial of a Pisot
  number.  The Zeckendorf numeration $Z$ is the simplest such example.  We
  define generalised equations of Z-Mahler type, and we show that if a
  sequence over a commutative ring is \emph{Z-regular}, then it is the
  sequence of coefficients of a series which is a solution of a Z-Mahler
  equation. Conversely, if the Z-Mahler equation is \emph{isolating}, then
  its solutions define Z-regular sequences. This is a generalisation of
  results of Becker and Dumas. We provide an example to show that there
  exist non-isolating Z-Mahler equations whose solutions do not define
  Z-regular sequences. Our proof yields a new construction of weighted
  automata that generate classical $q$-regular sequences. Our results can
  be generalised to numeration systems generated by recurrences whose
  characteristic polynomial is the minimal polynomial of a Pisot number.
\end{abstract}

\section{Introduction}\label{sec:intro}

Christol’s theorem is a bridge between algebraicity and finite state
computability. It characterises the algebraicity of formal power series
$f(x)$ with coefficients in a finite field of characteristic $p$, as being
those where the $n$-th coefficient of $f(x)$ is computed by a deterministic
finite automaton using the base-$p$ representation of $n$. Such sequences
are called \emph{$p$-automatic.}  Christol's theorem is firmly anchored in
algebra, and its beauty lies in the fact that it connects algebraicity of
series over finite fields to automata theory via numeration. However, it
falls short of being a complete characterisation of automatic sequences, as
it only characterises $q$-automaticity when $q=p^k$ is a prime power.
Another limitation is that it is only concerned with sequences over finite
alphabets.

On the automata side, both these shortcomings are overcome with the notion
of \emph{regularity}, introduced by Allouche and Shallit in
\cite{Allouche-Shallit-1992}. Just as automatic sequences are generated by
deterministic automata, regular sequences are generated by \emph{weighted}
automata.  By a weighted automaton, we mean a nondeterministic automaton
where each transition is labeled from an alphabet $B$ and in addition
carries a weight from a ring $R$; see Section~\ref{sec:weighted-automata}
for a precise definition. A word~$w$ on~$B$ is assigned a weight by the
automaton, namely, the sum of weights of all possible paths with $w$ as
label. Now sequences are again generated via numeration, but their values
are weights in $R$. Automatic sequences are regular, and if a regular
sequence is generated by a weighted automaton with weights from a finite
ring, then it is automatic \cite{Allouche-Shallit-1992}.

On the algebraic side, algebraic equations over the field $\mathbb{F}_q$ are
replaced by \emph{$k$-Mahler} equations over a commutative ring $R$,
defined as follows.  Let $k ⩾ 2$ be any natural number.  Define the linear
operator $Φ: R⟦ x⟧→ R⟦ x⟧$ as $Φ(f(x))=f(x^k)$.  Let $A_i(x) ∈ R[x]$ be
polynomials. The equation
\begin{align}\label{eq:q-Mahler}
  P(x,y) = ∑_{i=0}^dA_i(x)Φ^i(y)=0
\end{align}
is called a \emph{$k$-Mahler equation}, and if $f∈ R⟦ x⟧$ satisfies
$P(x,f(x))=0$, then it is called \emph{$k$-Mahler}, or a solution of
$P(x,y)=0$. The Mahler equation $P(x,y)$ is \emph{isolating} if
 the coefficient $A_0(x)$  of  $f(x)$ in~\eqref{eq:q-Mahler} is the constant polynomial~$1$.    Becker \cite{Becker-1994}, and Dumas\cite{Dumas-1993}
established links between Mahler equations and regularity by showing that,
for natural $k ⩾ 2$, $k$-regular sequences satisfy $k$-Mahler equations, and
conversely that solutions of isolating $k$-Mahler equations equations are
$k$-regular.  If $q$ is a power of a prime, and $R = \mathbb{F}_q$, then it can be
shown that the properties of being algebraic and $q$-Mahler coincide. Thus
the notion of a $k$-Mahler series is a generalisation of an algebraic
series, and the result of Becker and Dumas can be seen as a generalisation
of Christol's theorem. Furthermore, with the correct definitions of
regularity and $k$-Mahler equations, the proof of this result is quite
similar to a proof of Christol's theorem.

The base-$k$ numeration system is a fundamental underpinning of the notion
of $k$-automaticity.  In \cite{Rigo-2000}, Rigo extended the notion of
automaticity to finite valued sequences generated by abstract numeration
systems $U$: the $n$-th term of a sequence is calculated using its base-$U$
expansion via a deterministic automaton. Allouche and Shallit's notion of
regularity extends to these numeration systems. Combining regularity as
defined by Allouche and Shallit, to the notion of $U$-automaticity as
defined by Rigo, we obtain U-regular sequences.  A well studied family of
such $U$ are given by $U=(u_n)_{n⩾0}$ where $u_0=1$ and where $(u_n)_{n⩾0}$
is an increasing sequence that is obtained using a linear recurrence, whose
characteristic polynomial is a Pisot polynomial; see
Section~\ref{sec:basic-notation} for definitions. In this paper we
establish a Christol-type theorem for the simplest such $U$, namely when
$U=Z=1,2,3,5,…$ is the \emph{Zeckendorf} numeration, generated by the
recurrence $u_{n+1}= u_n+u_{n-1}$, with $u_{-1}=1$, $u_{-2}=0$.  We
introduce $Z$-Mahler equations; our main result is the $Z$-analogue of the
result of Becker and Dumas. One direction, namely that $Z$-regular
sequences are solutions of $Z$-Mahler equations, can be proved similarly to
the base-$k$ case; we do this in Section~\ref{sec:automata-to-Z-Mahler} for
the Zeckendorf numeration.  The proof of the converse direction is
nontrivial; our main result is the following.
\begin{theorem} \label{thm:main-Z-intro}
  Let $R$ be a commutative ring, let $Z$ be the Zeckendorf 
  numeration, and let $P(x,y)=0$ be an isolating Z-Mahler equation. If
  $f = ∑_{n ⩾ 0}f_nx^n$ is a solution of $P(x,f(x))=0$, then there is a
  weighted Z-automaton $\mathcal{A}$ which generates~$f$.
\end{theorem}
Furthermore, we believe that this theorem generalises to any Pisot
numeration.

Although the statement of Theorem~\ref{thm:main-Z-intro} is similar to that
of Becker and Dumas, the proof techniques are different, because in
$k$-numeration, appending a $0$ equals multiplication by~$k$, but this is
no longer true for~Z. This has several consequences. In particular
\emph{Cartier} operators (see \cite[Section~12.2]{Allouche-Shallit}), which
are used extensively in the base-$k$ proof, no longer have the required
properties in the world of Z-numeration.

To overcome the obstacle that we cannot use Cartier operators, our first
contribution is a direct construction of a weighted automaton computing the
coefficients, in the classical setting of the base-$k$ numeration, of the
solution of an isolating $k$-Mahler equation.  It can be used to obtain the
result of Becker and Dumas, or indeed one direction of Christol's theorem,
as in Corollary~\ref{cor:Christol}.  One novelty of our direct approach is
that it reveals the uniform structure of weighted automata that generate
$k$-regular sequences: they are all obtained from a \emph{universal}
weighted $k$-automaton, defined in Section~\ref{sec:def-automaton}; we name
it thus because each $k$-regular sequence can be obtained by instantiating
coefficients in it.
\begin{theorem} \label{thm:main-thm-1}
  Let $k ⩾ 2$ be an integer.  There exists a universal weighted automaton,
  such that for an explicit isolating $k$-Mahler equation $P(x,y)$ over the
  commutative ring~$R$ and an initial condition $f_0∈ R$, the coefficients
  of $P(x,y)$ and $f_0$ instantiate the weights on~$\mathcal{A}$, and the
  resulting weighted automaton $\mathcal{A}_{P,f_0}$ generates the solution
  $f$ of $P(x,y)=0$ with $f(0)=f_0$.
\end{theorem}

We state the relevant definitions and prove Theorem~\ref{thm:main-thm-1} in
Section~\ref{sec:equation-to-automaton}. There we also bound the
cardinality of the state set of~$\mathcal{A}$ as a function of~$P$'s
\emph{exponent} and \emph{height} where the exponent of~$P$ is the minimal
$d$ in~\eqref{eq:q-Mahler} and the height of~$P$ is the maximal degree of
the polynomials $A_i$.

The main obstacle to extending Theorem~\ref{thm:main-thm-1} to the Zeckendorf numeration is that the natural
replacement $ϕ :ℕ → ℕ$ of the map $n ↦ kn$, which we define in
Section~\ref{sec:fun-phi}, is not linear, and this latter property is
crucial in the construction of our automata in Theorem~\ref{thm:main}.
Using work of Frougny and
Solomyak~\cite{Frougny-1992,Frougny-Solomyak-1992} on the normalisation of
addition, we show that the \emph{linearity defect} of~$\phi$ can be
computed by a deterministic automaton.  From this, we introduce in
Section~\ref{sec:operator} a linear operator~$Φ: R⟦ x⟧ → R⟦ x⟧$ on series
which plays the role of the Frobenius operator $x ↦ x^k$ for $k$-Mahler
equations, and with it we define Z-Mahler equations
in~\eqref{eq:Phi-power}.
   
The proof of Theorem~\ref{thm:main-Z-intro} is obtained by combining the
proof of Theorem~\ref{thm:main-thm-1}, along with the tracking of the
linearity defect of~$ϕ$.  Furthermore in Theorem~\ref{thm:main-Z} we bound
the state set of~$\mathcal{A}$ as a function of~$P$ and the golden
ratio~$φ$ although we leave open the question of whether this bound is
tight. Conversely in Corollary~\ref{cor:regular-is-mahler} we see that
Z-regular series are solutions of a Z-Mahler equation.  Becker and Dumas
give examples of non-isolating Mahler equations whose solutions are not
$k$-regular. Similarly, we show in Section~\ref{sec:non-regular} that the
restriction to isolating equations is needed for our construction.  Finding
a complete characterisation of regular sequences is still open in general,
although it has been solved in \cite{Bell-Chyzak-Coons-Dumas} for
$R=\mathbb{C}$. See also \cite{Coons-Spiegelhofer} for a general exposition
of regular sequences, as well as recent developments. At the moment we
do not know how to extend the work of \cite{Bell-Chyzak-Coons-Dumas} to obtain a
complete characterisation of Z-regular sequences; what they do is based on roots of polynomials over $\mathbb C$; we do not have the right setting for this.

We describe the motivation for this paper, originating from the following
two questions. First, this paper arose out of a desire to understand
whether there was a connection between two descriptions of certain
substitutional fixed points as projections of a more regular structure in
higher dimension. The first result concerns the realisation of Pisot
substitutional tilings as a \emph{model set}.  A model set is a projection
of a pseudo-lattice in $ℝ^n ×  G$ into $ℝ^n$, via a
window in a locally compact Abelian group $G$. A \emph{Meyer} set is a
finite set of translates of a model set. The importance of these sets is
shown in Meyer's work \cite[Theoremm V Chapter VII, Theoremm IV, page
48]{Meyer-1972} and their relevance to substitutions was made clear by Lee
and Solomyak \cite{Lee-Solomyak}, who proved that a substitutional tiling
shift has pure discrete spectrum if and only if the discrete set of control
points for the tiling are a Meyer set. The second result is Furstenberg's
theorem \cite{Furstenberg-1967}, which tells us that algebraic functions in
$\mathbb{F}_q⟦ x⟧$ are projections, onto one dimension, of Laurent series expansions
of rational functions in $\mathbb{F}_q(x,y)$. Combining Furstenberg's theorem with
Christol's theorem \cite{Christol-1979, Christol-1980}, we conclude that
these projections are codings of length-$q$ substitutional fixed points.

Second, we were also motivated by the fact that solutions of equations of
Mahler type have provided a useful source of transcendental
numbers. Indeed, Mahler started this line of investigation by showing that
if $α$ is algebraic with $0<|α|<1$, and if $f(z)$ is the solution of the
$2$-Mahler equation $Φ(f(x))= f(x) -x$, then $f(α)$ is
transcendental~\cite{Mahler-1982}. This approach has been greatly
generalised and is now known as Mahler's method, see the book
\cite{Nishioka} by Nishioka devoted to the foundations of this area, as
well as the more recent survey article \cite{Adamczewski} by Adamczewski.
It would be interesting to investigate whether analogous results exist for
Z-Mahler equations.

This paper is organised as follows.  In
Section~\ref{sec:weighted-automata-automatic} we set up notation, and
define weighted automata, regular sequences, and Mahler equations. For the
case of base-$k$ numeration, $k ∈ ℕ$, we describe the correspondence
between $k$-regular sequences and $k$-Mahler equations. In Section
\ref{sec:def-automaton}, we define a universal weighted automaton, and in
Theorem~\ref{thm:main} we show that, by varying the weights, we can
generate any solution of any isolating $k$-Mahler equation.  In Corollary
\ref{cor:Christol} we recover Christol's theorem in the special case where
the ring equals~$\mathbb{F}_q$, with $q$ a power of a prime. We then focus on the
Zeckendorf numeration. In Section~\ref{sec:fun-phi} we define 
a function $n ↦ ϕ(n)$ and show in Corollary~\ref{cor:autom-defect} that the
non-linearity of $ϕ$ can be calculated.  In Section~\ref{sec:operator} we
define Z-Mahler equations. By merging the ideas behind the weighted
automata generating solutions of $k$-Mahler equations and the automaton
computing the linearity defect of $ϕ$, we show in Theorem~\ref{thm:main-Z}
that any solution of an isolating Z-Mahler equation is Z-regular. In the
standard Section~\ref{sec:automata-to-Z-Mahler}, we show that Z-regular
sequences define solutions of Z-Mahler equations. In
Section~\ref{sec:Dumas} we give a result analogous to that of Dumas for
$k$-numeration, slighting relaxing the notion of Z-isolating which
guarantees Z-regularity. Finally in Section~\ref{sec:non-regular} we give
an example of a non-isolating Z-Mahler equation which has non-regular
solutions. We end by describing some open problems.

\section{Mahler equations and weighted automata}
\label{sec:weighted-automata-automatic}

\subsection{Basic notation}\label{sec:basic-notation}

We will work with numeration systems $U=(u_n)_{n⩾0}$ with $u_0=1$, where
$(u_n)_{n⩾ 0}$ is an increasing sequence of integers and where each natural
number can be represented using strings of symbols from a digit set~$B$,
i.e., for each $n$ there is a $k$ and $b_k,…, b_0$ from $B$ such that
$n=∑_{i=0}^k b_iu_i$.  If the sequence $U$ is defined by a linear
recurrence with an irreducible characteristic polynomial whose roots
consist of exactly one real number~$\beta$ with $\beta>1$, while
$|\beta|<1$ for all other roots, we say that $U$ is a \emph{Pisot}
numeration system.  For these numeration systems, every natural number has
a \emph{canonical} representation, which is the greatest representation for
the lexicographic ordering, and we use $(n)_ U$ to denote this canonical
representation. Also, given a word on the digit set, we use $[w]_U$ to
denote the natural number that has $w$ as a (possibly non-canonical)
representation in that system. Despite these definitions, we are permitted
to sometimes pad $(n)_U$ with leading zeros, eg, when we have to compare
the expansions of several integers at a time.

In this article we only work with base-$k$ numerations and Pisot
numerations.  Amongst Pisot numerations, we focus on the \emph{Zeckendorf}
numeration, defined in Section~\ref{sec:Zeckendorf}, as we prefer to give
explicit constructions of the objects that concern us. However, our work in
Section~\ref{sec:Mahler-Zeckendorf} applies to any Pisot numeration.

In Section~\ref{sec:weighted-automata-automatic}, we work with base-$k$
numerations, with $U = (k^n)_{n ⩾ 0}$, whose digit set is
$\{0,1, …, k-1\}$, and the canonical representation $(n)_k$ is that for
which the most significant digit is non-zero.  In
Section~\ref{sec:Mahler-Zeckendorf}, we work with the Zeckendorf
numeration, $Z = (F_n)_{n ⩾ 0}$ defined by the Fibonacci numbers, where the
digit set is $\{0,1\}$ and where the canonical representation $(n)_Z$ has
no consecutive occurrences of the digit~$1$; see
\cite[Section~3.8]{Allouche-Shallit} for a summary.

\subsection{Automata} \label{sec:weighted-automata}

In this section we recall the notion of a weighted automaton. 
  We use
many variants of automata.  The main result (Theorem~\ref{thm:main-Z}) of
this paper is phrased using weighted automata, but we also use other kinds
automata, such as deterministic automata with an output function or even
classical automata.  The inputs of automata are words, that is, sequences
of symbols from an alphabet, and each numeration system~$U$ associates with
each integer~$n$ a word, its canonical representation~$(n)_U$, that can be
fed as input to an automaton.  This allows automata to deal with integers
instead of words.

We assume the reader to be familiar with the basic notions concerning
automata; otherwise we refer the reader to~\cite{Sakarovitch09} for a
complete introduction. An \emph{automaton} $ℬ$ is a tuple $⟨ S,B, Δ, I, F⟩$
where $S$ is the state set, $B$ is the input alphabet, $Δ$ is the
transition relation and $I$ and $F$ are the sets of initial and final
states.  A transition $(p,b,q) ∈ Δ$ is a labelled directed edge between
states and is written $p \trans{b} q$. A word $w = b_1 ⋯ b_n$ is
\emph{accepted} if there is a sequence $q_0 \trans{b_1} q_1 ⋯ q_{n-1}
\trans{b_n} q_n$ of consecutive transitions such that $q_0 ∈ I$ and $q_n ∈
F$.  The automaton is \emph{deterministic} if $I=\{s_0\}$ is a singleton,
and if $p \trans{b} q$ and $p \trans{b} q'$ in~$Δ$ implies $q = q'$.  In
that case, the relation~$Δ$ is a function $Δ: S × B → S$; it can be
extended to a function from~$S × B^*$ to~$S$ by setting $Δ(q,ε) = q$ and
$Δ(s, b_1 ⋯ b_n) ≔ Δ(Δ(s, b_1 ⋯ b_{n-1}), b_n)$.  If the set~$F$ of final
states is replaced by an \emph{output function} $τ: S → A$, where $A$ is
the \emph{output alphabet}, then the automaton defines the function from
$B^*$ to~$A$ which maps the word~$w$ to $τ(Δ(s_0,w))$.  Given a numeration
system~$U$, a sequence $(a_n)_{n ⩾ 0}$ is called \emph{$U$-automatic} if
there is a deterministic automaton $⟨ S,B, Δ,\{s_0\}, τ⟩$ such that $a_n =
τ(Δ(s_0,(n)_U))$ for each $n ⩾ 0$.  For the numeration systems that we
study in this article, this definition is equivalent to requiring the
existence of an automaton such that $a_n = τ(Δ(s_0,w))$ whenever $n=[w]_U$,
for each $n ⩾ 0$.  If $(n)_U$ is read starting with the most significant
digit, we say that $(a_n)_{n ⩾ 0}$ is obtained in \emph{direct reading},
otherwise we say that it is obtained in \emph{reverse reading}. If
$(a_n)_{n ⩾ 0}$ is $U$-automatic, then we will also say that $f(x) ≔ ∑_{n ⩾
  0}a_n x^n$ is \emph{$U$-automatic}.

If $U$ is the classical base-$k$ numeration, $U$-automatic sequences are
called $k$-automatic, and these sequences have been extensively studied
\cite{Allouche-Shallit}. In~\cite{Shallit-1988} Shallit studied more
general $U$-automatic sequences; see also work by Allouche
\cite{Allouche-1992}, Rigo~\cite{Rigo-2000} and Maes and
Rigo~\cite{Maes-Rigo-2002}.

Let $R$ be a commutative ring.  A \emph{weighted} automaton~$\mathcal{A}$ with
weights in $R$ is a tuple $⟨ S,B,Δ,I,F ⟩$, where $S$ is a finite state set,
$B$ is an alphabet, $I:S → R$ and $F:S → R$ are the functions that assign
to each state an initial and a final weight and $Δ : S × B × S → R$ is a
function that assigns to each transition, i.e., to each labelled edge, a
weight. A transition $(s,b,s')$ such that $Δ(s,b,s') = r ∈ R$ is written $s
\trans{b:r} s'$.  A \emph{path}~$γ$ in~$\mathcal{A}$ is a finite sequence $s_0
\trans{b_1:r_1} s_1, s_1 \trans{b_2:r_2} s_2, ⋯ ,s_{n-1} \trans{b_n:r_n}
s_n$ of consecutive transitions.  The label of such a path is the word $w =
b_1 ⋯ b_n$ and the path is written $s_0 \trans{w:r} s_n$ where $r = r_1 ⋯
r_n$.  This notation is consistent with the notation $s \trans{b:r} s'$ for
transitions since a transition can be viewed as a path of
length~$1$.  The \emph{weight} $\weight_{\mathcal{A}}(γ)$ of the path~$γ$ is the
product $I(s_0)r F(s_n) = I(s_0) r_1 ⋯ r_n F(s_n)$.  Furthermore, the
weight of a word $w ∈ B^*$ is the sum of the weights of all paths with
label~$w$ and it is denoted $\weight_{\mathcal{A}}(w)$, i.e.,
\begin{displaymath}
  \weight_{\mathcal{A}}(w) = ∑_{γ = s_0 \trans{w:r} s_n} \weight_{\mathcal{A}}(γ).
\end{displaymath}

In our figures, non-zero initial and final weights are given over small
incoming and outgoing arrows. Missing transitions implicitly have zero
weight. If a state has zero initial weight, it will not have an initial
arrow, likewise for states with final weight zero.

Recall the two-element Boolean semiring $\mathbb{B} = \{0,1\}$, where the
sum and the product are $\max$ and~$\min$ respectively.  A
\emph{non-deterministic} automaton $⟨ S,B,Δ,F ⟩$ is a weighted automaton
where $R = \mathbb{B}$.
  
Let $\mathcal{A}$ be a weighted automaton with $B=\{0,…,k-1\}$ and let $U$
be a numeration system with digit set $\{0, 1, … , k-1\}$.  Define the
sequence $(a_n)_{n ⩾ 0}$ by $a_n ≔ \weight_{\mathcal{A}}((n)_{U})$, where
$(n)_{ U}$ is read starting with the most significant digit. Then we say
that the sequence $(a_n)_{n ⩾ 0}$ and the generating function
$f(x) ≔ ∑_{n⩾ 0}a_nx^n$ are \emph{$U$-regular}, generated by~$\mathcal{A}$.
If $U$ is the base-$k$ numeration, then a weighted automaton that generates
a $k$-regular sequence is called a weighted $k$-automaton. The notion of a
$k$-regular sequence was introduced by Allouche and Shallit, that it is
equivalent to this definition is shown in \cite[Theorem
2.2]{Allouche-Shallit-1992}. Rigo studied automatic sequences for abstract
numeration systems, using automata with output in \cite{Rigo-2000}. With
Maes in \cite{Maes-Rigo-2002}, he defined $U$-regularity, generalising the
approach introduced in \cite{Allouche-Shallit-1992}, obtaining some of the
equivalences of \cite[Theorem 2.2]{Allouche-Shallit-1992}, based mainly on
the notion of a kernel, but not the one which our approach is based on,
namely \cite[Theorem 2.2 (e)]{Allouche-Shallit-1992}. The notion of
$U$-regular sequences via the kernel is also investigated by Charlier,
Cisternino and Stipulanti in \cite{CCS}.

\begin{example}\label{ex:reg-not-aut}\normalfont
  In Figure~\ref{fig:weighted0} we give a weighted 2-automaton with weights
  in~$ℕ$ that generates the sequence $(a_n)_{n ⩾ 0}$ where $a_n$ is the
  number of occurrences of the digit $1$ in~$(n)_2$.  As the state~$s$ is
  the only state with a nontrivial initial weight, and the state~$t$ is the
  only state with a nontrivial final weight, then, to compute the $n$-th
  term, one sums the weights of all paths with label~$(n)_2$ from the state
  $s$ to the state $t$.  Note that there are as many such paths as the
  number of occurrences of the digit~$1$ in~$(n)_2$.  Since each such path
  has weight~$1$, the sum is the number of~$1$'s in~$(n)_2$.
  \begin{figure}[htbp]
  \begin{center}
  \begin{tikzpicture}[->,>=stealth',semithick,auto,inner sep=3pt]
    \tikzstyle{every state}=[minimum size=14]
    \node[state] (q00) at (0,0) {$s$};
    \node (q00-in) at (-1,0) {} ;
    \node (q01-out) at (4,0) {} ;
    \node[state] (q01) at (3,0) {$t$};
    \path (q00-in) edge node {$\tcr{1}$} (q00);
    \path (q01) edge node {$\tcr{1}$} (q01-out);
    \path (q00) edge[out=120,in=60,loop] node {$\begin{array}{c}
                                                  \tcb{0}{:}\tcr{1} \\
                                                  \tcb{1}{:}\tcr{1}
                                                \end{array}$} (q00);
    \path (q00) edge node {$\tcb{1}{:}\tcr{1}$} (q01) ;
    \path (q01) edge[out=120,in=60,loop] node {$\begin{array}{c} 
                                                  \tcb{0}{:}\tcr{1} \\
                                                  \tcb{1}{:}\tcr{1}
                                                \end{array}$} (q01);
  \end{tikzpicture}
  \end{center}
  \caption{A weighted automaton that generates the $(a_n)_{n ⩾ 0}$,
    where $a_n$ is the number of~$1$'s in~$(n)_2$.  The weight of an edge
    is given in red, and the blue numbers are the digits we read in
    $(n)_2$.}
  \label{fig:weighted0} 
\end{figure}
\end{example}

Note that the sequence in Example~\ref{ex:reg-not-aut} is not automatic, as
it has entries from an infinite alphabet. If the weights of the automaton
of Figure~\ref{fig:weighted0} are considered as elements of~$\mathbb{F}_2$
instead of elements of~$ℕ$, then $a_n = \weight_{\mathcal{A}}((n)_2)$ is
the number of~$1$'s in $(n)_2$ mod~$2$. The sequence $(a_n)_{n ⩾ 0}$ is
thus the Thue-Morse sequence, which is automatic.

\subsubsection{Matrix representations of weighted automata}
\label{sec:matrix-weighted}

Let $\mathcal{A}$ be a weighted automaton and let $n$ be its number of
states.  Then $\mathcal{A}$ can also be represented by a triple
$\tuple{I,μ,F}$ where $I$ is a row vector over~$R$ of dimension $1 × n$,
$μ$ is a morphism from~$B^*$ into the set of $n × n$-matrices over~$R$,
with the usual matrix multiplication, and $F$ is a column vector of
dimension $n × 1$ over~$R$.  The vector~$I$ is the vector of initial
weights, the vector~$F$ is the vector of final weights and, for each
symbol~$b$, $μ(b)$ is the matrix whose $(p,q)$-entry is the weight~$r ∈ R$
of the transition $p \trans{b:r} q$.  Note that with this matricial
notation, $\weight_{\mathcal{A}}(w) = Iμ(w)F$.

\begin{example}
  \normalfont
  The weighted automaton pictured in Figure~\ref{fig:weighted0}
  is represented by $\tuple{I,μ,F}$ where $I= (1,0)$, $F = (0, 1)^t$ and
  the morphism~$μ$ from~$\{0,1\}^*$ to the set of $2 × 2$-matrices
  over~$\mathbb{F}_2$ is given by
  \begin{displaymath}
    μ(0) =
    \left(
      \begin{array}{cc}
        1 & 0 \\
        0 & 1
      \end{array}
    \right)
    \quad\text{and}\quad
    μ(1) = 
    \left(
      \begin{array}{cc}
        1 & 1 \\
        0 & 1
      \end{array}
    \right).
  \end{displaymath}
\end{example}

The following result can be proved as in
\cite[Theorem~2.10]{Allouche-Shallit-1992}.  It will be used in
Section~\ref{sec:non-regular} to show that the solution of a non-isolating
Z-Mahler equation is not Z-regular.
\begin{lemma} \label{lem:asymptotic-bound}
  If the sequence $(a_n)_{n ⩾ 0}$ is complex valued and $U$-regular,
  then there is a positive constant~$c$ such that $a_n= O(n^c)$.
\end{lemma}

% \begin{example}
%   The weighted automaton pictured in Figure~\ref{fig:weighted1} is
%   represented by $\tuple{π,μ,ν}$ where $π= (f_0,0, 0)$, $ν = (1, 0, 0)^t$
%   and the morphism~$μ$ from~$\{0,1\}^*$ to the set of $3 × 3$-matrices is
%   given by
%   \begin{displaymath}
%     μ(0) =
%     \left(
%       \begin{array}{ccc}
%         α_{1,0} & 0 & 0 \\
%         α_{1,2}  & α_{1,1} &  α_{1,0} \\
%         0 & α_{1,3} & α_{1,2}
%       \end{array}
%     \right)
%     \quad\text{and}\quad
%     μ(1) = 
%     \left(
%       \begin{array}{ccc}
%         α_{1,1} & α_{1,0} & 0 \\
%         α_{1,3}  & α_{1,2} &  α_{1,1} \\
%         0 & 0 & α_{1,3} 
%       \end{array}
%     \right).
%   \end{displaymath}
% \end{example}

\subsubsection{From weighted automata to automatic sequences}

It is known that any weighted automaton with weights in a finite ring
defines an automatic sequence; see for example
\cite[Thm~16.1.5]{Allouche-Shallit}. Nevertheless, for completeness we
include a proof, which is a slight generalisation of the power set
construction, and which yields a deterministic automaton from a
non-deterministic one. Recall that a direct reading automaton links
automatic sequences to morphic sequences, i.e, codings of fixed points of
substitutions, as shown by Rigo in \cite{Rigo-2000}.

\begin{proposition} \label{prop:weighted-to-automatic}
  Let $\mathcal{A}$ be a weighted automaton with $B=\{0,\dots ,k-1\}$ and
  with weights in a finite commutative ring $R$. Then, for each of direct
  and reverse reading, there exists a $k$-deterministic automaton~$ℬ$ with
  input alphabet $B$, with initial state $s_0$ and with output $τ: S → R$
  such that for each word~$w$, the equality
  $\weight_{\mathcal{A}}(w) = τ(Δ(s_0, w))$ holds.
\end{proposition}

\begin{proof}
  Let $n$ be the number of states of $\mathcal{A}$, and let $\tuple{I,μ,F}$ be a
  matrix representation of dimension~$n$ of the weighted automaton~$\mathcal{A}$ as
  given in Section~\ref{sec:matrix-weighted}.  The weight of a word
  $w ∈ B^*$ is by definition $Iμ(w)F$.
  
  Consider the automaton~$ℬ$ defined as follows. Its state set~$S$ is the
  finite set of all row vectors of dimension~$1 × n$ over~$R$. Its initial
  state is the vector~$I$.  Set $Δ(q,b)≔ qμ(b)$ for each row vector~$q$
  and each digit~$b$.  Finally define the function $τ:S ↦ R$ as $τ(q) = qF$
  for each $q ∈ S$.  It is now routine to check that the automaton~$ℬ$
  defines, in direct reading, the same automatic sequence as the weighted
  automaton.
  
  A reverse deterministic automaton is obtained similarly by taking $S$ to
  be the set of column vectors of dimensions $n × 1$ over~$R$, by taking
  $F$ as initial state and, setting $Δ(q,b) ≔ μ(b)q$.
\end{proof}

Note that Proposition~\ref{prop:weighted-to-automatic} implies the
following.  If $\mathcal{A}$ is a weighted automaton with
$B=\{0,… ,k-1\}$ and with weights in a finite commutative ring~$R$, and
$U$ is a numeration system with digit set~$B$, then the $U$-regular
sequence generated by~$\mathcal{A}$ is $U$-automatic.

\subsection{Robustness of weighted automata}

As we saw in Proposition~\ref{prop:weighted-to-automatic}, there is a
direct link from sequences generated by weighted automata using a
numeration system $ U$ to $ U$-regular sequences. The purpose of this
section is to show that the class of $ U$-regular sequences is closed under
the Cauchy product as soon as addition in $ U$ can be realised by an
automaton.  Obtaining this result is made easier by the use of weighted
automata.  Looking ahead, it is particularly interesting when Christol's
theorem does not hold as for the Zeckendorf numeration system.

An \emph{unambiguous} automaton is a weighted automaton over the ring~$ℤ$
of integers such that each weight, including initial and final weights, is
either $0$ or~$1$, and such that the weight of each word is either $0$
or~$1$.  The first condition implies that the weight of each path is in
$\{0, 1\}$.  The second condition implies that for each word~$w$, there is
at most one path labelled by~$w$ having a positive weight.  If there is
such a path, the word is said to be \emph{accepted}.

Let $u$ and $v$ be two words over alphabets $A$ and~$B$ respectively such
that $|u| = |v|$.  We denote by $u ⊗ v$ the word~$w$ over the alphabet $A ×
B$ such that $|w| = |u| = |v|$ and $w_i = (u_i,v_i)$.  An unambiguous
automaton \emph{realises addition} in a numeration system if it accepts all
words of the form $(m)_U \otimes (n)_U \otimes (m+n)_U $ for non-negative
integers $m$ and~$n$, where the expansions $(m)_U$ and~$(n)_U$ have been
possibly padded with leading zeros to have the same length as $(m+n)_U$,
i.e., we momentarily relax our notation and use $(m)_U$ to refer to any
representation of~$m$.  For example, the automaton pictured in
Figure~\ref{fig:addition2} realizes addition in base~$2$.

\begin{figure}[htbp]
  \begin{center}
  \begin{tikzpicture}[initial text=,->,>=stealth',semithick,auto,inner sep=3pt]
    \tikzstyle{every state}=[minimum size=10]
    \node[state,initial left,accepting below] (q0) at (0,0) {$0$};
    \node[state] (q1) at (4,0) {$1$};
    \path (q0) edge[out=120,in=60,loop] node {$
      \begin{array}{c}
        (0,0,0) \\
        (1,0,1) \\
        (0,1,1)
      \end{array}$} (q0);
  \path (q0) edge[bend left=15] node {$(0,0,1)$} (q1);
  \path (q1) edge[out=120,in=60,loop] node {$
      \begin{array}{c}
        (1,0,0) \\
        (1,1,1) \\
        (1,1,1)
      \end{array}$} (q1);
  \path (q1) edge[bend left=15] node {$(1,1,0)$} (q0);
  \end{tikzpicture}
  \end{center}
  \caption{An automaton recognising addition base-$2$. A string over
    $\{0,1\}^3$, is accepted in direct reading if and only if it equals
    $(m)_2 \otimes (n)_2 \otimes (m+n)_2$. }
  \label{fig:addition2} 
\end{figure}

Recall that the \emph{Cauchy product}, or \emph{convolution} of two series
$f(x) = ∑_{n ⩾ 0}{f_nx^n}$ and $g(x) = ∑_{n ⩾ 0}{g_nx^n}$ is the series
$h(x) = ∑_{n ⩾ 0}{h_nx^n}$ where $h_n = ∑_{i+j = n}f_ig_j$ for each
$n ⩾ 0$.

In the literature, eg \cite{Peltomaki-Salo-2022}, a numeration where
addition is realised by an automaton is called \emph{addable}. This
condition appears as a condition in the following.  The
  next theorem is classical
for integer bases and the proof is usually carried out through the
\emph{kernel} of the sequence, see, eg,
\cite[Theorem~3.1]{Allouche-Shallit-1992}.  The proof we provide is
different and extends better to our setting based on weighted automata.
\begin{theorem} \label{thm:Cauchy-product}
  Suppose that there is an unambiguous automaton realising addition in the
  numeration system~$ U$. If the sequences $f_1 = (f_{1,n})_{n ⩾ 0}$ and
  $f_2 = (f_{2,n})_{n ⩾ 0}$ are generated by weighted automata in~$ U$,
  then the Cauchy product of $f_1$ and~$f_2$ is also generated by a
  weighted automaton in~$ U$.
\end{theorem}
\begin{proof}
  Let $B$ be the digit alphabet of the numeration system~$U$.  Suppose
  that the unambiguous automaton~$\mathcal{A}= ⟨ Q, B^3,Δ, I,F ⟩$ realises
  addition in~$ U$.  Let the sequences $f_1$ and~$f_2$ be generated by the
  weighted automata $ℬ_1= ⟨ S_1,B,Δ_1, I_1, F_1 ⟩$ and
  $ℬ_2= ⟨ S_2,B,Δ_2, I_2, F_2 ⟩$ respectively.  We construct a new weighted
  automaton~$ℬ$ whose state set is $Q × S_1 × S_2$.  The initial
  (respectively final) weight of a state $(q,s_1,s_2)$ of~$ℬ$ is
  $I_1(s_1)I_2(s_2)$ (respectively $F_1(s_1)F_2(s_2)$) if $q$ is initial
  (respectively final) in~$\mathcal{A}$, and $0$ otherwise.  Its transition
  set~$Δ$ is given by
  \begin{displaymath}
    Δ = 
    \left\{ (q,s_1,s_2) \trans{b:α_1α_2} (q',s'_1,s'_2) : ∃ b_1,b_2 ∈ B
      \quad 
      \begin{array}{rl}
        q \trans{b_1,b_2,b} q'   & \text{in } \mathcal{A} \\
        s_1 \trans{b_1:α_1} s'_1 & \text{in } ℬ_1 \\
        s_2 \trans{b_2:α_2} s'_2 & \text{in } ℬ_2 
      \end{array}
    \right\}.   
  \end{displaymath}
  A path $(q,s_1,s_2) \trans{w:α} (q',s'_1,s'_2)$ in~$ℬ$ corresponds first
  to a choice of a path $s \trans{w,u_1,u_2} s'$ in~$\mathcal{A}$ giving two
  words $u_1$ and~$u_2$ such that $[w]_U = [u_1]_U + [u_2]_U$, and second to a
  choice of two paths $s_1 \trans{u_1:α_1} s'_1$ and $s_2 \trans{u_2:α_2}
  s'_2$ in $ℬ_1$ and~$ℬ_2$ such that $α = α_1α_2$.  From this remark and
  the choices of initial and final weights in~$ℬ$, it follows that the
  weight of a word~$w$ is the sum, over all decompositions  $[w]_U = [u_1]_U +
  [u_2]_U$ of the products $α_1α_2$ of the weights $α_1$ and~$α_2$ of $u_1$
  and~$u_2$.  Note that the commutativity of the ring is here essential.
\end{proof}

\subsection{$k$-Mahler equations and weighted automata}
\label{sec:equation-to-automaton}

\subsubsection{From weighted automata to $k$-Mahler equations}
\label{sec:automaton-to-Mahler}

Given a weighted $k$-automaton generating $(a_n)_{n ⩾ 0}$, there is a
standard technique to obtain a Mahler equation for which $∑_{n ⩾ 0} a_nx^n$
is a solution, similar to the techniques described in the proof of
Christol's theorem in \cite[Theorem~12.2.5]{Allouche-Shallit}.
We illustrate with a simple example.
\begin{example} \label{ex:automaton-to-equation}
  Consider again the Thue-Morse sequence $(a_n)_{n ⩾ 0}$, whose weighted
  automaton is given in Figure~\ref{fig:weighted0}, but where here we
  consider the weights in $\mathbb{F}_2$, so that the resulting sequence is
  automatic. We interpret the state~$t$ to be the formal power series
  $t(x)= ∑_{n ⩾ 0} a_n x^n$ in~$\mathbb{F}_2⟦ x⟧$, where
  $a_n = \weight_{\mathcal{A}}((n)_2) ∈ \mathbb{F}_2$. Similarly the
  state~$s$ corresponds to the series $s(x)= ∑_{n ⩾ 0} b_n x^n$
  in~$\mathbb{F}_2⟦ x⟧$, whose coefficients would be generated by the
  automaton pictured in Figure~\ref{fig:weighted2}.  As we read $(n)_2$
  starting with the most significant digit, we have,
  \begin{alignat}{2} \label{eq:recursive}
    a_{2n} &= a_n & \quad\text{ and }\quad a_{2n+1} &= a_n ⊕ b_n \\
    b_{2n} &= b_n & \quad\text{ and }\quad b_{2n+1} &= b_n.\notag 
   \end{alignat}
  where the symbol~$⊕$ denotes the sum in~$\mathbb{F}_2$.
  \begin{figure}[htbp]
  \begin{center}
  \begin{tikzpicture}[->,>=stealth',semithick,auto,inner sep=3pt]
    \tikzstyle{every state}=[minimum size=14]
    \node[state] (q00) at (0,0) {$s$};
    \node (q00-in) at (-1,0) {} ;
    \node (q00-out) at (0,-1) {} ;
    \path (q00-in) edge node {$\tcr{1}$} (q00);
    \path (q00) edge node {$\tcr{1}$} (q00-out);
    \path (q00) edge[out=120,in=60,loop] node {$\begin{array}{c}
                                                  \tcb{0}{:}\tcr{1} \\
                                                  \tcb{1}{:}\tcr{1}
                                                \end{array}$} (q00);
  \end{tikzpicture}
  \end{center}
  \caption{A weighted automaton that generates the series $s(x)$.}
  \label{fig:weighted2} 
  \end{figure}

  Since $\mathbb{F}_2⟦ x⟧$ has characteristic~$2$, then
  \eqref{eq:recursive} implies that
  \begin{equation} \label{eq:recursive-2}
    t(x) = (1+x) t(x^2) + x s(x^2) \text{ and } s(x)= (1+x) s(x^2).
  \end{equation}
  Again using that $\mathbb{F}_2⟦ x⟧$ has characteristic~$2$, we have
   \begin{equation} \label{eq:recursive-3}
    t(x^2) = (1+x)^2 t(x^4) + x^2 s(x^4) \text{ and } s(x^2)= (1+x^2) s(x^4).
  \end{equation}

  Now multiplying \eqref{eq:recursive-2} by~$x$ and substituting \eqref{eq:recursive-3} in~\eqref{eq:recursive-2}, we obtain
  \begin{align*}
    xt(x) &=  (x+x^2+x^3+x^4)t(x^4) +(x^2+x^3)s(x^4) \\
          &=(1+x)^3t(x^4) +x^2(1+x)s(x^4) +(1+x^4)t(x^4) \\
          &=(1+x)t(x^2) + (1+x^4)t(x^4),
  \end{align*}
  i.e., the Thue-Morse power series is a solution of the $2$-Mahler equation
  $P(x,y)= xy+(1+x)Φ(y) +(1+x^4)Φ^2(y)  = 0$ where $Φ$ is the
  operator which maps each series $f(x)$ to~$f(x^2)$.
\end{example}

\subsubsection{Reformulating isolating $k$-Mahler equations}
\label{sec:reformulating-mahler}

Given an isolating $k$-Mahler equation
$P(x,y) = y - ∑_{i=1}^dA_i(x)Φ^i(y)$, we write
\begin{equation} \label{eq:coefficients}
  A_i(x) = ∑_{j=0}^h α_{i,j}x^j
  \qquad
  \text{for $1 ⩽ i ⩽ d$},
\end{equation}
where $α_{i,j}∈ R$ for each $i,j$.  Let $f(x) = ∑_{n⩾0}f_nx^n$ satisfy
$P(x,f(x))=0$.  To find the coefficients~$f_n$ of~$f$, we obtain
from~\eqref{eq:coefficients} that
\begin{align} \label{eq:condensed-recurrence-brut}
  f_n = ∑_{\begin{smallmatrix}
             \ell k^i+j=n \\
             1 ⩽ i ⩽ d, \; 0 ⩽ j ⩽ h
           \end{smallmatrix}}
           α_{i,j}f_\ell =     ∑_{
             \ell k^i+j=n 
                     }
           α_{i,j}f_\ell \, ,
\end{align}
where the last equality follows if we set $α_{i,j}=0$ for $i,j$ outside the
prescribed bounds. Therefore we can drop the superfluous constraints on the
indices.

Note that if we set $n=0$ in~\eqref{eq:condensed-recurrence-brut}, we get
the following equation satisfied by the coefficient~$f_0 = f(0)$
\begin{align} \label{eq:compatible}
   f_0=  \left(∑_{i=0}^hα_{i,0} \right) f_0.
\end{align}
Therefore, if it were to be the case that $f_0=f(0)$ for some solution~$f$
of $P(x,f(x))=0$, Equality~\eqref{eq:compatible} must hold.  In this case
we say that $f_0$ is \emph{$P$-compatible}. Conversely, if $f_0$ is
$P$-compatible, then, again using~\eqref{eq:condensed-recurrence-brut},
there is a unique series~$f$ such that $P(x,f(x))=0$ and $f(0)=f_0$.
Clearly, if $∑_{i=0}^hα_{i,0}=1$, then any $f_0 ∈ R$ is $P$-compatible.
Also, if $R$ is an integral domain, then $\bigl(∑_{i=0}^hα_{i,0}\bigr)f_0 =
f_0$ is equivalent to either $∑_{i=0}^hα_{i,0}=1$ or $f_0=0$.

In~\eqref{eq:condensed-recurrence-brut} we have reduced solving a
functional equation to a linear problem, to which weighted automata are
well suited.  The identity~\eqref{eq:condensed-recurrence-brut} motivates
the definition of the automaton associated to a Mahler equation that we
will use in Section~\ref{sec:def-automaton}, and it is also what makes the
proof of Proposition~\ref{prop:key} work.

\subsubsection{From isolating $k$-Mahler equations to weighted automata}
\label{sec:def-automaton}

In this section, we define a weighted automaton which directly computes the
coefficients of the solution of a $k$-Mahler equation. Note that we will
avoid the need for Cartier operators $\Lambda_i$, $0 ⩽ i ⩽ k-1$
(\cite[Definition~12.2.1]{Allouche-Shallit}); this is relevant, because
when we move to more general numeration systems in
Section~\ref{sec:Mahler-Zeckendorf}, the appropriate generalisation of
Cartier operators do not enjoy some of the properties that the
operators~$Λ_i$ satisfy.  We mention here that the notion of a kernel for
$U$-regular sequences has been nicely developed in~\cite{Maes-Rigo-2002, CCS}.

This automaton is given in two steps.  First we describe a \emph{universal}
weighted automaton with an infinite number of states.  The structure of
this universal automaton is fixed and does not depend on the ring~$R$;
furthermore it can accommodate any $k$-Mahler equation~$P$. Given such an
equation, which has finitely many non-zero coefficients, only finitely many
states of this universal automaton are needed to compute the solution, and
the weights of its transitions are given by the coefficients occurring
in~$P$.

Define the state set~$S$
\begin{displaymath}
  S ≔ \{ s_{i,j}: 0 ⩽ i < ∞ \text{ and } 0 ⩽ j < ∞ \}.
\end{displaymath}
Let $B = \{0,1, \dots , k-1\}$ be the input alphabet.  Let
$(α_{i,j})_{i⩾ 1, j⩾ 0}$ and $f_0$ be elements of the commutative ring~$R$,
and define the transition set~$Δ$
\begin{align*}
  Δ  ≔ {} & \left\{ s_{i,j} \trans{b:1} s_{i+1,kj+b} : b ∈  B,
              0 ⩽ i,j \right\} \\
    {} ∪ {} & \left\{ s_{i,j} \trans{b:α_{i+1,kj+b-\ell}} s_{0,\ell} : b ∈ B,
              0 ⩽ i,j,\ell  \text{ and } 0 ⩽ kj+b-\ell \right\}.
\end{align*}
We set the initial and final weights $I$ and $F$ as 
\begin{displaymath}
  I(s_{i,j}) ≔
  \begin{cases}
    f_0 \text{ if $j=0$} \\
    0 \text{ otherwise,  }
  \end{cases}
  \quad\text{and}\quad
  F(s_{i,j}) ≔
  \begin{cases}
    1 \text{ if $i = j =0$} \\
    0 \text{ otherwise. }
  \end{cases}
\end{displaymath}
Then we call the automaton $\mathcal{A}≔⟨ S,B,Δ,I,F ⟩$ the \emph{universal weighted
  $k$-automaton associated to $(α_{i,j})_{i⩾1, j⩾0}$}.   If needed, we can
specify that $\mathcal{A}$ depends on the choice of initial condition $f_0$.

We think of this weighted $k$-automaton as universal because only the edge
weights depend on $(α_{i,j})_{i⩾ 1, j⩾ 0}$.  Let $α_{i,j}$ for $i⩾ 1$ and
$j⩾ 0$ be the coefficient of $x^jΦ^i(y)$ in the expression
$∑_{i=1}^dA_i(x)Φ^i(y)$ which is viewed as a polynomial in~$x$ and~$Φ$.
The following lemma shows that if there are finitely many non-zero
coefficients~$α_{i,j}$, only finitely many states of the universal
automaton are useful.  Useful means here that they occur in a path with
non-zero weight.  Note that such a path must end in~$s_{0,0}$ since this
state is the only one with non-zero final weight. Lemma~\ref{lem:truncate}
also provides an explicit upper bound of the number of useful states.
\begin{lemma} \label{lem:truncate}
  Let $P$ be an isolating $k$-Mahler equation with exponent~$d$ and
  height~$h$.  If either $i ⩾ d$ or $j ⩾ \frac{h}{k-1}$, then the
  state~$s_{i,j}$ in the universal weighted $k$-automaton does not occur in
  a path with non-zero weight.
\end{lemma}
\begin{proof}
  We claim that any path with a non-zero weight starting from a state
  $s_{i,j}$ with either $i ⩾ d$ or $j ⩾ h/(k-1)$ ends in a state
  $s_{i',j'}$ with either $i' ⩾ d$ or $j' ⩾ h/(k-1)$.  The claim proves the
  statement of the lemma because $s_{0,0}$ is the only state with a
  non-zero final weight and this state cannot be reached from~$s_{i,j}$
  with a non-zero weight path.

  It suffices, of course, to prove the claim when the path is of
  length~$1$, that is, equals a transition $s_{i,j} \trans{b:α} s_{i',j'}$
  in the universal weighted $k$-automaton associated to
  $(α_{i,j})_{i⩾ 1, j⩾ 0}$ where $α_{i,j}=0$ if either $i ⩾ d+1$ or $j⩾ h+1$.

  Let us first suppose that $i ⩾ d$.  The transition $s_{i,j} \trans{b:1}
  s_{i+1,kj+b}$ satisfies the claim because $i+1 ⩾ d$.  The weight
  $α_{i+1,kj+b-ℓ}$ of the transition $s_{i,j} \trans{b:α_{i+1,kj+b-ℓ}}
  s_{0,ℓ}$ is zero for the same reason.

  Let us suppose now that $j ⩾ h/(k-1)$.  The transition $s_{i,j}
  \trans{b:1} s_{i+1,kj+b}$ satisfies the claim because $kj+b ⩾ j ⩾
  h(k-1)$.  Finally suppose $s_{i,j} \trans{b:α_{i+1,kj+b-ℓ}} s_{0,ℓ}$.
  Either $kj+b-ℓ ⩾ h+1$ and the weight $α_{i+1,kj+b-ℓ}$ is zero, or $kj+b-ℓ
  ⩽ h$ and $ℓ$ satisfies $ℓ > kj-h ⩾ kh/(k-1) -h = h/(k-1)$, and this
  completes the proof.
\end{proof}

In view of Lemma~\ref{lem:truncate}, and because $s_{0,0}$ is the only
state with a non-zero final weight, we define
the  weighted automaton associated to a  $k$-Mahler equation~$P$ and
initial condition~$f_0$ as follows.

Let $k ⩾ 2$ be a natural number. Let
$P(x,y) = y - ∑_{i=1}^dA_i(x)y^{k^i}$ be an isolating $k$-Mahler
equation whose coefficients $A_i(x) ∈ R[x]$ are as
in~\eqref{eq:coefficients}.  Set $\tilde{h} ≔ ⌈\frac{h}{k-1}⌉-1$.  If we
truncate the universal weighted $k$-automaton associated to
$(α_{i,j})_{i ⩾ 1, j ⩾ 0}$ and $f_0$ to the state and transition sets
\begin{displaymath}
  S ≔ \{  s_{i,j}: 0 ⩽ i ⩽ d-1 \text{ and } 0 ⩽ j ⩽\tilde{h}  \}
\end{displaymath}
and
\begin{align*}
  Δ  ≔ {} & \left\{ s_{i,j} \trans{b:1} s_{i+1,kj+b} : b ∈ B
            \begin{array}{c}
              0 ⩽ i ⩽ d-2 \\
              0 ⩽ kj+b ⩽ \tilde{h}
            \end{array} \right\} \\
    {} \cup {} & \left\{ s_{i,j} \trans{b:α_{i+1,kj+b-\ell}} s_{0,\ell} : b ∈ B
              \begin{array}{c}
                0 ⩽ i ⩽ d-1 \\
                0 ⩽ j,\ell ⩽ \tilde{h} \\
                0 ⩽ kj+b-\ell ⩽ h
              \end{array} \right\},
\end{align*}
the automaton $\mathcal{A} ≔ ⟨S,B,Δ,I,F⟩$ is called the \emph{weighted
  automaton associated to $P$ and~$f_0$}, so
$\mathcal{A} = \mathcal{A}(P,f_0)$.

In other words, the weighted automaton associated to~$P$ is a restriction
of the universal $k$-automaton to a subgraph which will contain all paths
in the universal automaton that have non-zero weight and end at~$s_{0,0}$.
Because of this, we will frequently relax some of the constraints on the
indices. Note also that while this definition does not require $f_0$ to be
$P$-compatible, the latter will be a condition for all our results.

\begin{example} \normalfont In Figure~\ref{fig:weighted1} we give the
  weighted automaton associated to the $2$-Mahler equation
  $f(x) = A_1(x)f(x^2)$ where $A_1(x)$ is the polynomial of degree~$3$
  $α_{1,0}+α_{1,1}x+α_{1,2}x^2+α_{1,3}x^3$. We will see in
  Proposition~\ref{prop:key} that setting $α_{1,0} = 1$ and $f_0$ in the
  ring~$R$, this automaton generates the unique solution
  $f(x) = ∑_{n ⩾ 0}f_nx^n ∈ R⟦ x⟧$ of the equation $f(x) = A_1(x)f(x^2)$
  such that $f(0) = f_0$.
  \begin{figure}[htbp]
  \begin{center}
  \begin{tikzpicture}[->,>=stealth',semithick,auto,inner sep=1.2pt]
  \tikzstyle{every state}=[minimum size=0.4]
  \node[state] (q00) at (0,0) {$s_{0,0}$};
  \node (q00-in) at (-1,0) {} ;
  \node (q00-out) at (0,-1) {} ;
  \node[state] (q01) at (3,0) {$s_{0,1}$};
  \node[state] (q02) at (6,0) {$s_{0,2}$};
  \path (q00-in) edge node {$\tcr{f_0}$} (q00);
  \path (q00) edge node {$\tcr{1}$} (q00-out);
  \path (q00) edge[out=120,in=60,loop] node {$\begin{array}{c}
                                                \tcb{0}{:}\tcr{α_{1,0}} \\
                                                \tcb{1}{:}\tcr{α_{1,1}}
                                              \end{array}$} (q00);
  \path (q00) edge[bend left=15] node {$\tcb{1}{:}\tcr{α_{1,0}}$} (q01) ;
  \path (q01) edge[bend left=15] node {$\begin{array}{c}
                                          \tcb{0}{:}\tcr{α_{1,2}} \\
                                          \tcb{1}{:}\tcr{α_{1,3}}
                                        \end{array}$} (q00) ;
  \path (q01) edge[out=120,in=60,loop] node {$\begin{array}{c} 
                                                \tcb{0}{:}\tcr{α_{1,1}} \\
                                                \tcb{1}{:}\tcr{α_{1,2}}
                                              \end{array}$} (q01);
  \path (q01) edge[bend left=15] node {$\begin{array}{c}
                                          \tcb{0}{:}\tcr{α_{1,0}} \\
                                          \tcb{1}{:}\tcr{α_{1,1}}
                                        \end{array}$} (q02) ;
  \path (q02) edge[bend left=15] node {$\tcb{0}{:}\tcr{α_{1,3}}$} (q01) ;
  \path (q02) edge[out=120,in=60,loop] node {$\begin{array}{c} 
                                                \tcb{0}{:}\tcr{α_{1,2}} \\
                                                \tcb{1}{:}\tcr{α_{1,3}}
                                              \end{array}$} (q01);
  \end{tikzpicture}
  \end{center}
  \caption{The weighted automaton for $f(x) = (α_{1,0}+α_{1,1}x+α_{1,2}x^2+α_{1,3}x^3)f(x^2)$.}
  \label{fig:weighted1} 
\end{figure}
\end{example}
 
\begin{example}\normalfont
  Figure~\ref{fig:weighted-2-3}  
  depicts the weighted automaton associated to a
  $2$-Mahler equation of exponent~$d=2$ and height~$h=3$, that is, the equation
  \begin{displaymath}
    f(x) = A_1(x)f(x^2) + A_2(x)f(x^4)
  \end{displaymath}
  where $A_1(x) = α_{1,0}+α_{1,1}x+α_{1,2}x^2+α_{1,3}x^3$ and
  $A_2(x) = α_{2,0}+α_{2,1}x+α_{2,2}x^2+α_{2,3}x^3$. Note how
  Figure~\ref{fig:weighted1} is a subgraph of
  Figure~\ref{fig:weighted-2-3}; consistent with the fact that they are
  both truncations of the universal weighted $2$-automaton.
\begin{figure}[htbp]
  \begin{center}
  \begin{tikzpicture}[->,>=stealth',semithick,auto,inner sep=1.2pt]
  \tikzstyle{every state}=[minimum size=0.4]
  \node[state] (q00) at (0,4) {$s_{0,0}$};
  \node (q00-in) at (-1,4) {} ;
  \node (q00-out) at (0.7,4.7){} ;
  \node[state] (q01) at (4,4) {$s_{0,1}$};
  \node[state] (q02) at (8,4) {$s_{0,2}$};
  \node[state] (q10) at (0,0) {$s_{1,0}$};
  \node (q10-in) at (-1,0) {} ;
  \node[state] (q11) at (4,0) {$s_{1,1}$};
  \node[state] (q12) at (8,0) {$s_{1,2}$};

  \path (q00-in) edge node {$\tcr{f_0}$} (q00);
  \path (q10-in) edge node {$\tcr{f_0}$} (q10);
  \path (q00) edge node {$\tcr{1}$} (q00-out);
  \path (q00) edge[out=140,in=80,loop] node {$\begin{array}{c}
                                                \tcb{0}{:}\tcr{α_{1,0}} \\
                                                \tcb{1}{:}\tcr{α_{1,1}}
                                              \end{array}$} (q00);
  \path (q00) edge[bend left=10] node {$\tcb{1}{:}\tcr{α_{1,0}}$} (q01);
  \path (q00) edge[bend left=10] node {$\tcb{0}{:}\tcr{1}$} (q10);
  \path (q00) edge[bend left=10] node[pos=0.75] {$\tcb{1}{:}\tcr{1}$} (q11);
  \path (q01) edge[bend left=10] node {$\begin{array}{c}
                                          \tcb{0}{:}\tcr{α_{1,2}} \\
                                          \tcb{1}{:}\tcr{α_{1,3}}
                                        \end{array}$} (q00);
  \path (q01) edge[out=120,in=60,loop] node {$\begin{array}{c} 
                                                \tcb{0}{:}\tcr{α_{1,1}} \\
                                                \tcb{1}{:}\tcr{α_{1,2}}
                                              \end{array}$} (q01);
  \path (q01) edge[bend left=10] node {$\begin{array}{c}
                                          \tcb{0}{:}\tcr{α_{1,0}} \\
                                          \tcb{1}{:}\tcr{α_{1,1}}
                                        \end{array}$} (q02);
  \path (q01) edge[bend left=10] node[pos=0.75] {$\tcb{0}{:}\tcr{1}$} (q12);
  \path (q02) edge[bend left=10] node {$\tcb{0}{:}\tcr{α_{1,3}}$} (q01);
  \path (q02) edge[out=120,in=60,loop] node {$\begin{array}{c} 
                                                \tcb{0}{:}\tcr{α_{1,2}} \\
                                                \tcb{1}{:}\tcr{α_{1,3}}
                                              \end{array}$} (q01);
  \path (q10) edge[bend left=10] node {$\begin{array}{c}
                                          \tcb{0}{:}\tcr{α_{2,0}} \\
                                          \tcb{1}{:}\tcr{α_{2,1}}
                                        \end{array}$} (q00);
  \path (q10) edge node[pos=0.75,swap] {$\tcb{1}{:}\tcr{α_{2,0}}$} (q01);
  \path (q11) edge[bend left=10] node[pos=0.25,xshift=6pt] {$\begin{array}{c}
                                                               \tcb{0}{:}\tcr{α_{2,2}} \\
                                                               \tcb{1}{:}\tcr{α_{2,3}}
                                                             \end{array}$} (q00);
  \path (q11) edge node[swap,xshift=-4pt] {$\begin{array}{c}
                                              \tcb{0}{:}\tcr{α_{2,1}} \\
                                              \tcb{1}{:}\tcr{α_{2,2}}
                                            \end{array}$} (q01);
  \path (q11) edge node[swap,pos=0.21,xshift=-4pt] {$\begin{array}{c}
                                          \tcb{0}{:}\tcr{α_{2,0}} \\
                                          \tcb{1}{:}\tcr{α_{2,1}}
                                        \end{array}$} (q02);
  \path (q12) edge[bend left=10] node[pos=0.15,xshift=4pt] {$\tcb{0}{:}\tcr{α_{2,3}}$} (q01);
  \path (q12) edge node[swap] {$\begin{array}{c}
                                  \tcb{0}{:}\tcr{α_{2,2}} \\
                                  \tcb{1}{:}\tcr{α_{2,3}}
                                \end{array}$} (q02);
  \end{tikzpicture}
  \end{center}
  \caption{The automaton for a 2-Mahler equation of exponent 2 and height 3.}
  \label{fig:weighted-2-3} 
\end{figure}
\end{example}

\begin{remark} \label{remark:leading-zero}\normalfont
  The aim is to generate a sequence $(f_n)_{n ⩾ 0}$ by feeding $(n)_k$ into
  a weighted $k$-automaton. We claim that the automaton associated to~$P$
  will generate the same sequence even if we allow leading zeros
  in~$(n)_k$. It is sufficient to show that $Iμ(0)= I$. The only states
  with non-zero initial weight are $s_{i,0}$. The transitions between these
  states with $b=0$ are $s_{i,0} \trans{0:α_{i+1,0}} s_{0,0}$ and
  $s_{i,0} \trans{0:1} s_{i+1,0}$ for $0 ⩽ i ⩽ d-1$. Thus
  \begin{displaymath}
    I\mu(0)= (\overbrace{f_0, \dots f_0}^{d } | 0 … 0)  
    \begin{pmatrix}
      \begin{array}{ccccc}
      α_{1,0} & 1 & 0 & \dots & 0 \\ 
      α_{2,0} & 0 & 1 & \dots & 0 \\
      \vdots \\
      α_{d-1,0} & 0 & 0 & \dots &1\\  
      α_{d,0} & 0 & 0 & \dots & 0\\  
      \end{array}  & \rvline & \mathbf{0} \\ \hline
      M_2
    \end{pmatrix}  = I,
  \end{displaymath}
  where $\mathbf{0}$ represents the appropriate dimension matrix all of
  whose entries are zero, and where we have used
  Property~\eqref{eq:compatible}, i.e., that
  $\left(∑_{i=1}^{d} α _{i,0}\right)f_0=f_0$.
\end{remark} 
Given a state $s$ in the automaton $\mathcal{A}(P,f_0)$, define 
\begin{displaymath}
  \weight_{\mathcal{A},s}^*(w) = ∑_{s_0 \trans{w:r} s}I(s_0) r,
\end{displaymath}
where the sum is taken over all possible paths that start at~$s_0$, are
labelled~$w$, and end at~$s$.  The following proposition states the main
property of the automaton associated to a given Mahler equation, showing
that it indeed computes the solution of the equation.

\begin{proposition} \label{prop:key}
  Let $R$ be a commutative ring, let $P(x,y) ∈ R[x,y]$ be an isolating
  $k$-Mahler equation of exponent~$d$ and height~$h$, and let
  $f = ∑_{n ⩾ 0}f_n x^n$ be a solution of $P(x,f(x))=0$.  If $\mathcal{A}$
  is the weighted automaton associated to $P(x,y)$ and $f_0 ∈ R$, if
  $w ∈ \{0,… , k-1\}^*$, and if $i$ and~$j$ are integers with $0 ⩽ i ⩽ d-1$
  and $0 ⩽ j ⩽ ⌈\frac{h}{k-1}⌉-1$, then
  \begin{displaymath}
    \weight_{\mathcal{A},s_{i,j}}^{*}(w) =
    \begin{cases}
      f_\ell  & \text{if $\ell k^i + j = [w]_k$} \\
      0    & \text{otherwise.}
    \end{cases}
  \end{displaymath}
\end{proposition}
\begin{proof}
  The proof is by induction on the length of the word~$w$.  If $w$ is the
  empty word~$ε$, then $ \weight_{\mathcal{A},s_{i,j}}^*(ε) = I(s_{i,j})$; this
  equals~$f_0$ if $j= 0$ and equals~$0$ otherwise, so the claim is proved
  for $w = ε$ since $[ε]_k = 0$.
  
  For $|w| > 0$, we consider first the easier case $i > 0$ and then the
  case $i = 0$.

  Let $i > 0$. In this case there is a unique incoming transition
  $s_{i-1,j'} \trans{b:1} s_{i,j}$ to~$s_{i,j}$ if $j ≡ b \mod k$ and
  $kj'+b = j$, and no incoming transition otherwise. In the latter case the
  statement follows trivially, so in what follows we assume that there is an
  incoming transition. Using the inductive hypothesis,
 %\begin{align*} 
  %   \weight_{\mathcal{A},s_{i,j}}^{*}(wb) & =   \weight_{\mathcal{A},s_{i-1,j'}}^{*}%(w) \\
    %                           & = f_\ell \text{ where $k^{i-1}\ell + j' = [w]_k$.} 
 % \end{align*}
  \begin{align*} 
     \weight_{\mathcal{A},s_{i,j}}^{*}(wb)  =   \weight_{\mathcal{A},s_{i-1,j'}}^{*}(w) 
                                = f_\ell,  
  \end{align*}
  where $k^{i-1}\ell + j' = [w]_k$.
  By multiplying $k^{i-1}\ell + j' = [w]_k$ by~$k$ and then adding~$b$, we
  obtain $k^i\ell + j = k[w]_k + b = [wb]_k$ and the statement of the
  proposition is proved in this case. Notice here that we use linearity of
  the map $m ↦ km$.
             
  Next we consider the case $i=0$.  We have
 \begin{align*}
   \weight_{\mathcal{A},s_{0,j}}^*(wb)
     & = ∑_{i',j'}α_{i'+1,k j'+b-j}\weight_{\mathcal{A},s_{i',j'}}^*(w) \\
     & = ∑_{(i',j',\ell):\,k^{i'}\ell + j' = [w]_k}α_{i'+1,kj'+b-j} f_\ell,
  \end{align*}
  where the first equality follows by the definition of the allowed
  transitions in~$\mathcal{A}$ and the second equality follows by the inductive
  hypothesis. Setting $i''=i'+1$ and $ℓ' = kj'+b-j$, the equation
  $k^{i'}\ell + j' = [w]_k$ becomes $k^{i''} \ell + ℓ' = [wb]_k-j$, and  we have
  \begin{align*}
  \weight_{\mathcal{A},s_{0,j}}^*(wb) & =   ∑_{k^{i''} ℓ + ℓ' = [wb]_k-j}α_{i'' ,ℓ'}f_\ell \\
                           & = f_{[wb]_k-j},
  \end{align*}
  where the last equality follows from \eqref{eq:condensed-recurrence-brut}. 
\end{proof}

We can now show that the weighted automaton associated to a Mahler equation
generates the desired solution.
\begin{theorem} \label{thm:main}
  Let $R$ be a commutative ring, and let $P(x,y)∈ R[x,y]$ be an isolating
  $k$-Mahler equation of exponent~$d$ and height~$h$.  Then, for any
  $w\in \{0, \ldots , k-1\}^*$ the automaton~$\mathcal{A}$ associated to
  $P$ and~$f_0$ satisfies
  \begin{displaymath}
    \weight_{\mathcal{A}}(w) = f_{[w]_k},
  \end{displaymath}
  where $f = ∑_{n ⩾ 0}f_nx^n$ is a solution of $P(x,f(x))=0$.
  
  Consequently if $R$ is finite, then there exists a deterministic
  automaton with at most $|R|^{ \left⌈\frac{h}{k-1}\right⌉ d}$ states that
  generates~$f(x)$.
\end{theorem}
\begin{proof}
  Given $P(x,y) = y- ∑_{i=1}^dA_i(x)y^{k^i}$, with $A_i$ defined by
  Equation~\eqref{eq:coefficients}, let $\mathcal{A}$ be the weighted
  automaton associated to $P$; by definition it has at most
  $⌈\frac{h}{k-1}⌉ d$ states. Let $w$ be any word with $[w]_k=n$. By
  Remark~\ref{remark:leading-zero}, $f_{[w]_k}=f_n$, so it suffices to
  check the statement for $w= (n)_k$.  By Proposition~\ref{prop:key}, we
  have $\weight_{\mathcal{A},s_{0,0}}^{*}(w)= f_n$. As $s_{0,0}$ is the
  only state with a non-zero final weight, and
  $\weight_{\mathcal{A},q}^{*}(w)=(Iμ(w))_q$, we have
  \begin{displaymath}
    \weight_{\mathcal{A}}(w)= Iμ(w)F = \weight_{\mathcal{A},s_{0,0}}^{*}(w)=f_n.
  \end{displaymath}
  The second statement follows by determinising, either in reverse or in
  direct reading, the automaton~$\mathcal{A}$, as in
  Proposition~\ref{prop:weighted-to-automatic}.
\end{proof}

For the case of automatic sequences over $R=\mathbb{F}_q$, with $q=p^k$
where $p$ is prime, our approach gives a novel proof of one implication in
Christol's theorem, and a construction of a deterministic automaton
generating $(f_n)_{n ⩾ 0}$, without the use of Cartier operators.
\begin{corollary}\label{cor:Christol}
  Let $R$ be the finite field~$\mathbb{F}_q$. If
  $∑_n f_n x^n ∈ \mathbb{F}_q⟦ x⟧$ is algebraic over $\mathbb{F}_q(x)$,
  then there is a deterministic automaton that generates $(f_n)_{n ⩾ 0}$.
\end{corollary}
\begin{proof}
  If $∑_n f_n x^n ∈ \mathbb{F}_q ⟦x⟧$ is algebraic, then standard
  techniques, eg, \cite[Lemma 12.2.3]{Allouche-Shallit}, imply that it is
  the root of an Ore polynomial, i.e., a polynomial of the form
  $P(x,y)= ∑_{i=0}^dA_i(x)y^{q^i}$.  If $f$ is the root of this Ore
  polynomial, then the series $g ≔ f/A_0$ is the root of the polynomial
  $Q(x,y) = ∑_{i=0}^dB_i(x)y^{p^i}$, with $B_0 = 1$ and
  $B_i = A_iA_0^{q^i-2}$.  Note that $Q$ is an isolating $q$-Mahler
  equation. By Theorem~\ref{thm:main} we can construct a weighted
  automaton~$ℬ$ that generates~$g$.  Then, we apply
  Theorem~\ref{thm:Cauchy-product}, which allows us to construct a weighted
  automaton $\mathcal{A}$ of the Cauchy product $f = A_0g$.  Now using
  Proposition~\ref{prop:weighted-to-automatic}, $\mathcal{A}$ can be
  determinised, to yield automata that can generate $(f_n)_{n ⩾ 0}$ in
  either reverse or direct reading.
\end{proof}

\section{Mahler equations for  Zeckendorf numeration}
\label{sec:Mahler-Zeckendorf}

In this section we investigate links between generalised Mahler equations
and weighted automata in Pisot numerations, and in particular the Zeckendorf numeration Z.  Our principal result
in this section, Theorem~\ref{thm:main-Z}, is a version of Becker's theorem,
which says that a solution of an isolating Z-Mahler equation, as defined
in  \eqref{eq:Z-Mahler-equation} below, has coefficients computed by a
weighted automaton reading integers in the Zeckendorf base. Conversely, any
such weighted automaton generates a solution of a Z-Mahler equation,
which is possibly non-isolating.  The gap between the two results is
similar to the one in~\cite{Becker-1994}.  The notion of Z-Mahler
equation that we introduce is based on a linear operator~$Φ$, defined on
power series, which is the analogue of the operator $f(x) ↦ f(x^k)$ in
base-$k$ numeration.  This operator~$Φ$ is defined thanks to a function~$ϕ:
ℕ\rightarrow ℕ$ which plays the role of the function $n ↦ kn$. The main
obstacle to our extension of Becker's result is that $ϕ$ is not linear.

\subsection{The Zeckendorf numeration}\label{sec:Zeckendorf}

Recall that the Fibonacci numbers satisfy the recurrence
$F_n=F_{n-1}+F_{n-2}$ with initial conditions $F_{-2}=0$ and $F_{-1}=1$.
The strictly increasing sequence $(F_n)_{n⩾0}$ defines the
\emph{Zeckendorf} numeration system: every natural number has a unique
expansion as $n= ∑_{i = 0}^kb_iF_i$ where $b_i ∈ \{0,1\}$ for each $i$ and
$b_i \times b_{i+1}=0$ for each $i$.  We write $(n)_Z ≔ b_k ⋯ b_0$.  Conversely,
given any finite digit set $B ⊂ ℤ$ and any word $w= w_k\dots w_0 ∈ B^+$,
let $[w]_Z$ denote the natural number~$n$ such that $n = ∑_{i =
  0}^kw_iF_i$. Note that $(n)_Z$ is the canonical Zeckendorf expansion
of~$n$, but conversely, the map $w ↦ [w]_Z$ can be applied to any word over
a finite digit set~$B ⊂ ℤ$.  Thorough descriptions of the properties of
addition in this numeration system can be found in Frougny's exposition
\cite[Chapter~7]{Lothaire}.

\begin{example}\label{ex:preliminary}
  The following weighted automaton generates $(a_n)_{n ⩾ 0}$, where $a_n$
  equals the number of representations of~$n$ as a sum of distinct
  Fibonacci numbers.  Note that each non-zero weight (in red in the figure)
  is equal to~$1$.  This implies that the weight of a path is either $0$
  or~$1$.  This weight is $1$ if the path starts and ends at state~$0$.
  The number of such paths labeled by $(n)_Z$ is exactly~$a_n$. For
  example, $8=[10000]_Z=[1100]_Z= [01011]_Z$, and there are three paths
  each with weight~$1$, namely
  $0 \trans{1} 0 \trans{0}0\trans{0} 0\trans{0}0 \trans{0} 0$,
  $0 \trans{1} 1 \trans{0}2\trans{0} 1\trans{0}2 \trans{0} 0$, and
  $0 \trans{1} 1 \trans{0}2\trans{0} 0\trans{0}0 \trans{0} 0$.
  Therefore $a_8=3$.
 
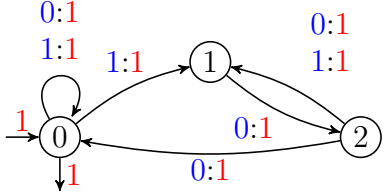
  \begin{figure}[htbp]
  \begin{center}
  \begin{tikzpicture}[->,>=stealth',semithick,auto,inner sep=3pt]
    \tikzstyle{every state}=[minimum size=14]
    \node[state] (q00) at (0,0) {$0$};
    \node (q00-in) at (-1,0) {} ;
    \node (q00-out) at (0,-1) {} ;
    \node[state] (q01) at (2,1) {$1$};
    \node[state] (q02) at (4,0) {$2$};
    \path (q00-in) edge node {$\tcr{1}$} (q00);
    \path (q00) edge node {$\tcr{1}$} (q00-out);
    \path (q00) edge[out=120,in=60,loop] node {$\begin{array}{c}
                                                  \tcb{0}{:}\tcr{1} \\
                                                  \tcb{1}{:}\tcr{1}
                                                \end{array}$} (q00);
    \path (q00) edge[bend left=13] node[pos=0.7] {$\tcb{1}{:}\tcr{1}$} (q01);
    \path (q01) edge[bend right=13] node[swap] {$\tcb{0}{:}\tcr{1}$} (q02);
    \path (q02) edge[bend right=13] node[swap] {$\begin{array}{c}
                                                \tcb{0}{:}\tcr{1} \\
                                                \tcb{1}{:}\tcr{1}
                                              \end{array}$} (q01);
    \path (q02) edge[bend left=10] node {$\tcb{0}{:}\tcr{1}$} (q00);
  \end{tikzpicture}
  \end{center}
 % \caption{A weighted automaton that computes the numbe}
  \label{fig:weighted} 
\end{figure}
\end{example}

\subsection{The function $ϕ$ and its almost-linearity}
\label{sec:fun-phi}

The following function is the analogue, in  a Pisot numeration, of
the map $n ↦ kn$ in base-$k$.  Define $ϕ: ℕ → ℕ$ as  $ϕ(n)≔ [w0]_U$, where $w=(n)_U$. We study   in particular the case where $U$ equals Z;
 There, $ϕ(F_n) = F_{n+1}$. By definition,
\begin{displaymath}
  ϕ \left(∑_{i=0}^kb_i{F_{i}}\right) ≔ ∑_{i=0}^kb_i{F_{i+1}}
\end{displaymath}
where $(n)_Z = b_k ⋯ b_0$. Furthermore, the above equality holds for
any word~$w$ over $\{0,1\}$ such that $[w]_Z=n$ even if there are
consecutive digits~$b_i$ such that $b_i = b_{i+1} = 1$.  The function~$ϕ$
and the map $w ↦ [w]_Z$ are linked by the equation $[wb]_Z = ϕ([w]_Z) + b$
for any word~$w$ over $\{0,1\}$ and any digit $b ∈ \{0,1\}$.

If $(m)_Z = b_k ⋯ b_0$ and $(n)_Z = c_k ⋯ c_0$ with possibly some leading
zeros, $m$ and~$n$ are said to have \emph{disjoint support} if
$b_i \times c_i = 0$ for $0 ⩽ i ⩽ k$. For example, the numbers $2$ and~$3$
have disjoint support, but the numbers $3$ and~$4$ do not.  If $(m)_Z$ and
$(n)_Z$ have disjoint support, then $ϕ(m+n) = ϕ(m)+ϕ(n)$. However, the
function~$ϕ$ is not linear, for example $ϕ(2) = 3 ≠ 4 =
ϕ(1)+ϕ(1)$. Nevertheless, $ϕ$ is \emph{almost} linear: In
Lemma~\ref{lem:almost-linear} we show that $ϕ(m+n)-ϕ(m)-ϕ(n)$ belongs to
$\{-1,0,1\}$ for each $m,n ⩾ 0$.  We write $ϕ^2$ for the function $ϕ ∘ ϕ$.
Note that if $(n)_Z=w$, then $(ϕ^2 (n))_Z=w00$ and~$(ϕ^2 (n )+ 1)_Z = w01$.
We give below the first few values of $ϕ(n)$ and $ϕ^2(n)+1$.

\begin{figure}
  \begin{displaymath}
    \begin{array}{r|cccccccccccccc}
    n & 0 & 1 & 2 & 3 & 4 & 5 & 6 & 7 & 8 & 9 & 10 & 11 & 12 & 13 \\ \hline
  ϕ (n) & 0 & 2 & 3 & 5 & 7 & 8 & 10 & 11 & 13 & 15 & 16 & 18 & 20 & 21 \\
  ϕ^2 (n) +1 & 1 & 4 & 6 & 9 & 12 & 14 & 17 & 19 & 22 & 25 & 27 & 30 & 33 & 35
    \end{array}
  \end{displaymath}
  \caption{The first few values of $ϕ(n)$ and $ϕ^2(n)+1$.}
  \label{fig:valuesofphi}
\end{figure}

The sets $A_1 = \{ ϕ(n) : n ⩾ 0\}$ and $A_2 = \{ ϕ^2 (n)+1 : n ⩾ 0\}$ form
a partition of~$ℕ$, as $A_1$ and $A_2$ are the sets of integers whose
Zeckendorf expansion ends with~$0$ and $1$ respectively. Let $φ$ denote the
golden ratio. Then Lemma~\ref{lem:floor-Cartier} below implies that $A_1$
and $A_2$ are translations of the Beatty sequences
$B_1 = \{ ⌊φ n⌋ : n ⩾ 1 \}$ and $B_2 = \{ ⌊φ^2n⌋ : n ⩾ 1 \}$ which are
known to form a partition of~$ℕ ∖ \{0\}$ as
$\frac{1}{φ} + \frac{1}{φ^2} = 1$.  It will allow us to bound the
non-linearity of~$ϕ$ in Lemma~\ref{lem:almost-linear}.

\begin{lemma} \label{lem:floor-Cartier}
  For $n ⩾ 0$ an integer, we have $ϕ(n) = ⌊φ n+φ-1⌋$ and $ϕ^2(n)
  = ⌊φ^2n+φ -1⌋$. 
\end{lemma}
  
The proof is an easy application of Binet's formula, which we include for
completeness.

\begin{proof}
  Let $\bar{φ} = -1/φ$ be the algebraic conjugate of~$φ$; we prove the
  first equality.  From $F_n = \frac{1}{\sqrt{5}}(φ^{n+2} - \bar{φ}^{n+2})$
  it follows that $φ F_n - F_{n+1} = - \bar{φ}^{n+2}$.

  Suppose that $n = ∑_{i=0}^k{F_{n_i}}$ for a sequence $n_0,…,n_k$ of
  integers such that $n_{i+1} > n_i+1$ for each integer $0 ⩽ i ⩽ k-1$.
  We have
  \begin{displaymath}
    φ n - ϕ(n) = ∑_{i=0}^k{(φ F_{n_i} - F_{n_i+1})}
    = -∑_{i=0}^k{\bar{φ}^{n_i+2}}
  \end{displaymath}
  We now bound the summation $∑_{i=0}^k{\bar{φ}^{n_i+2}}$.  Since
  $\bar{φ}$ is negative, odd powers of~$\bar{φ}$ are negative
  while even powers are positive, so
  \begin{displaymath}
    -\bar{φ}^2 = \frac{\bar{φ}^3}{1-\bar{φ}^2} = 
    ∑_{n ⩾ 1}{\bar{φ}^{2n+1}} <
    ∑_{i=0}^k{\bar{φ}^{n_i+2}}
    < ∑_{n ⩾ 1}{\bar{φ}^{2n}} =
    \frac{\bar{φ}^2}{1-\bar{φ}^2} = -\bar{φ}.
  \end{displaymath}
  It follows that
  \begin{displaymath}
    0 = \bar{φ}+φ-1 < φ n+φ-1-ϕ(n) < \bar{φ}^2+φ-1 = 1
  \end{displaymath}
  Since $ϕ(n)$ is an integer, this completes the proof of the first
  equality.  The proof of the second one uses $φ^2F_n - F_{n+2} = -
  \bar{φ}^{n+2}$ and follows the same lines.
\end{proof}

We define the \emph{linearity defect~$δ$} of~$ϕ$ by $δ(m,n) =
ϕ(m+n)-ϕ(m)-ϕ(n)$ for  non-negative integers $m$ and~$n$.  It turns out
that although $ϕ$ is not linear, its linearity defect is small, as the
following lemma shows.

\begin{lemma} \label{lem:almost-linear}
  For natural numbers $m,n$, we have $-1 ⩽ δ(m,n) ⩽ 1$.
\end{lemma}
We remark that this proof can be generalised to show that if $U$ is Pisot,
then there is a $k=k(U)$ such that $|δ(m,n)|⩽ k$ for $m,n ∈ ℕ$.
\begin{proof}
  We apply the relation $ϕ(k) = ⌊φ k+φ-1⌋$, obtained
  in Lemma~\ref{lem:floor-Cartier},  to $k=m+n$, $k=m$ and $k=n$:
  \begin{align*}
   φ(m+n)+φ -2 & < ϕ(m+n) < φ(m+n)+φ-1,\\
    1-φ m-φ    & < -ϕ(m)  <  2-φ m-φ \mbox{ and } \\
    1-φ n-φ    & < -ϕ(n)  <  2-φ n-φ .
  \end{align*}
  Adding these three relations gives $-φ < ϕ(m+n)-ϕ(m)-ϕ(n) < 3-φ$, and the
  statement follows as $1 < φ < 2$.
\end{proof}

\subsection{Regularity of the linearity defect}

We will sometimes use $\bar{1}$ to denote $-1$, in particular when $-1$ is
an element of an alphabet.  Set $B ≔ \{0,1\}$ and $\bar{B} ≔
\{\bar{1},0,1\}$.  Lemma~\ref{lem:almost-linear} tells us that $δ(m,n) ∈
\bar{B}$.  Next, we obtain a more precise version of
Lemma~\ref{lem:almost-linear}, as we will later need to use an automaton
that, given $m>n$, computes $δ(m-n,n)$.

The following theorem is proved in \cite[Proposition~7.3.11,
Chapter~7]{Lothaire}.  Recall that a set of words~$K$ is \emph{regular} if
there is a deterministic automaton $ℬ = ⟨ S,B, Δ, \{s_0\}, F ⟩$ such that
$Δ(s_0, w)$ is in an accepting state if and only if $w ∈ K$.

\begin{theorem}\label{thm:regularity-0}
  Let $U$ be a Pisot numeration system. For any finite set $C ⊂ ℤ$, the set
  $\{ w ∈ C^*: [w]_U = 0\}$ is regular.
\end{theorem}

Fix a finite set $C$ of digits. Let $u = u_m ⋯ u_0$ and $v = v_n ⋯ u_0$ be
words in $C^*$.  By padding with leading zeros, it can be assumed that
$m = n$.  We denote respectively by $u ⊞ v$ and $u ⊟ v$ the two words
$(u_m + v_m) ⋯ (u_0 + v_0)$ and $(u_m - v_m) ⋯ (u_0 - v_0)$.  Note that
this definition is slightly ambiguous as padding $u$ and~$v$ with more
leading zeros yields a word with more leading zeros.  However, it does not
hurt as we are only interested in $[u ⊞ v]_Z$ and $[u ⊟ v]_Z$.  Note that
both operations are associative.  It is easily verified that
\begin{equation}\label{eq:plus}
  [u ⊞ v]_Z = [u]_Z + [v]_Z \quad\text{and}\quad [u ⊟ v]_Z = [u]_Z - [v]_Z.
\end{equation}

Note that equality $u0 ⊞ v0 = (u ⊞ v)0$ holds for words $u,v ∈ C^*$ and
similarly for the operation~$⊟$.  The equality
$[(m)_Z0 ⊞ (n)_Z0]_Z = ϕ(m) + ϕ(n)$ also holds for integers $m,n ⩾ 0$;
similarly for the operation~$⊟$. Let $L = \{ (m)_Z : m ⩾ 0 \} ⊂ B^*$ be
the set of canonical expansions.

In the remainder of this section we discuss the regularity of~$δ$, by which
we mean that there is a deterministic automaton that outputs $δ(m-n,n)$
when given $(m)_Z$ and $(n)_Z$; see Corollary~\ref{cor:autom-defect} for
the precise statement.  For $b ∈ \bar{B}$, define
\begin{displaymath}
  X_b ≔ \{ (m)_Z ⊟ (n)_Z : m ⩾ n ⩾ 0 \text{ and } δ(m-n,n) = b \}.
\end{displaymath}

\begin{proposition} \label{pro:defect-reg}
  The sets $X_{\bar{1}}$, $X_0$ and $X_1$ are pairwise disjoint and regular.
\end{proposition}

The fact that the three sets $X_{\bar{1}}$, $X_0$ and $X_1$ are pairwise
disjoint means that whenever $(m)_Z ⊟ (n)_Z = (m')_Z ⊟ (n')_Z$, then
$δ(m-n,n) = δ(m'-n',n')$.
We remark that using similar methods, we can also show that the sets $\{
(m)_Z ⊞ (n)_Z : m,n ⩾ 0 \text{ and } δ(m,n) = b  \}$, for $b ∈ \bar{B}$, are
also pairwise disjoint and regular.
The proof of Proposition~\ref{pro:defect-reg} is based on the following key
lemma which gives a characterisation of $δ(m-n,n) =b$.
\begin{lemma} \label{lem:key-sum}
  Let $b ∈ \bar{B}$. Then for $m ⩾ n ⩾ 0$
  \begin{displaymath}
    δ(m-n,n) = b \iff
    ∃ w ∈ L \mbox{ such that }
    \begin{cases}
      [w ⊟ \left((m)_Z ⊟ (n)_Z\right)]_Z = 0 \\
      [w0 ⊟ \left((m)_Z0 ⊟ (n)_Z0\right)]_Z = -b.
    \end{cases}
  \end{displaymath}
\end{lemma}
\begin{proof}
  Suppose that $δ(m-n,n) = b$ and let $w$ be equal to $(m-n)_Z$.  Then,
  repeated use of~\eqref{eq:plus} gives
  \begin{align*}
    [w ⊟ \left ((m)_Z ⊟ (n)_Z\right)]_Z &=[w]_Z - [(m)_Z ⊟ (n)_Z)]_Z \\
       & = [w]_Z - [(m)_Z]_Z + [(n)_Z)]_Z= [w]_Z - m + n = 0,
  \end{align*}
  and similarly
  $[w0 ⊟ \left( (m)_Z0 ⊟ (n)_Z0\right) ]_Z = [w0]_Z - ϕ(m) + ϕ(n) = -δ(m-n,n)
  = -b$.

  Conversely, suppose that there exists $w ∈ L$ satisfying both required
  equalities. Again using~\eqref{eq:plus}, the first equality $[w ⊟ \left( (m)_Z ⊟ (n)_Z\right)]_Z = 0$
  implies that $[w]_Z = m-n$ and $[w0]_Z = ϕ(m-n)$.  Combined with $w ∈ L$,
  it shows that $w = (m-n)_Z$.  The second equality yields
  $[w0 ⊟ \left((m)_Z0 ⊟ (n)_Z0\right)]_Z = [w0]_Z - ϕ(m) + ϕ(n) =
  -δ(m-n,n)$. Now the result follows by Lemma~\ref{lem:almost-linear}. 
\end{proof}

Now we come to the proof of Proposition~\ref{pro:defect-reg}.
\begin{proof}[Proof of Proposition~\ref{pro:defect-reg}]
  By Lemma~\ref{lem:key-sum}, the value of $δ(m-n,n)$ is determined by the
  word $(m)_Z ⊟ (n)_Z$.  This shows that the three sets $X_{\bar{1}}$, $X_0$
  and~$X_1$ are pairwise disjoint.  It remains to show that they are
  regular.  We will prove that $X_0$ is regular, the proofs for $X_{\bar{1}}$
  and~$X_1$ being similar.
  
  The construction of a non-deterministic automaton accepting~$X_0$ is
  based on the statement of Lemma~\ref{lem:key-sum}.  Such an automaton
  reads $(m)_Z ⊟ (n)_Z$, and non-deterministically ``guesses'' the word $w
  ∈ L$ given by Lemma~\ref{lem:key-sum}.  In particular, it checks that
  both equalities $ [w ⊟ ((m)_Z ⊟ (n)_Z)]_Z = 0$ and $ [(w0 ⊟ ((m)_Z ⊟
  (n)_Z))0]= 0$ hold.  This means that each of its transitions reads a
  symbol $b' ∈ \bar{B} $ of $(m)_Z ⊟ (n)_Z$, and guesses a symbol~$b'' ∈ B$
  of~$w$.  Non-deterministic guessing is done by having two transitions
  reading a given symbol~$b'$, one making the guess $b'' = 0$ and another one
  making the guess $b'' = 1$.  Note that $(m)_Z ⊟ (n)_Z$ has entries from
  $\bar{B} $ and $w$ has entries from $B$, so $w ⊟ ((m)_Z ⊟ (n)_Z)$ has
  entries from $C ≔ \{\bar{1}, 0, 1, 2\}$.  In order to check both required
  equalities, the automaton simulates a deterministic automaton accepting
  $Y = \{ u ∈ C^* : [u]_Z = 0 \}$ for $C = \{\bar{1}, 0, 1,2 \}$.  If the
  read symbol is~$b'$ and the guessed symbol for~$w$ is~$b''$, the automaton
  simulates the automaton for~$Y$ with entry~$b''-b'$.  The automaton also
  stores in its states the last guessed~$b''$ to check that there are no
  consecutive digits~$1$s in~$w$.
  
  We give below a more formal description of a non-deterministic
  automaton~$ℬ$ accepting~$X_0$.  Let $C=\{\bar{1}, 0, 1, 2\}$ and let $\mathcal{C} =
  ⟨ Q, C,Δ_C, \{q_0\},F⟩$ be a deterministic automaton accepting $\{ u ∈
  C^* : [u]_Z = 0 \}$, whose existence is guaranteed by
  Theorem~\ref{thm:regularity-0}.  The input alphabet of~$ℬ$ is $\bar{B}$
  since $ℬ$ is fed with words of the form $(m)_Z ⊟ (n)_Z$ for $m,n ⩾ 0$.
  The state set of~$ℬ$ is the set~$Q × \{0,1\}$ and its unique initial
  state is~$(q_0,0)$.  Its transition set~$Δ$ is defined as 
  \begin{align*}
    Δ  {} ≔ {} & \{ (p,0) \trans{b'} (q,0) : Δ_C(p,-b') = q \} \\
       {} ∪ {} & \{ (p,1) \trans{b'} (q,0) : Δ_C(p,-b') = q \} \\
       {} ∪ {} & \{ (p,0) \trans{b'} (q,1) : Δ_C(p,1-b') = q \}.
  \end{align*}
  Each transition reads a symbol~$b'$ of the input word $(m)_Z ⊟ (n)_Z$ and
  guesses a symbol $b'' ∈ B$ for the word~$w$ (see Lemma~\ref{lem:key-sum}).
  The corresponding symbol of $w ⊟ \left((m)_Z ⊟ (n)_Z\right)$ is thus $b'' -
  b'$.  The second component of each state is equal to the last guessed~$b$
  to check that $w$ does not contain two consecutive~$1$s.
  The set of final states is
  \begin{displaymath}
    F' ≔ \{ (q,x) : q ∈ F \text{ and }
                    ∃ q' ∈ F \text{ with } (q,x) \trans{0} (q',0) \},
  \end{displaymath}
  so that both statements are verified.  The first condition $q ∈ F$
  guarantees the first equality $[w ⊟ ((m)_Z ⊟(n)_Z)]_Z = 0$. The second
  condition guarantees the second equality $[(w ⊟ ((m)_Z ⊟ (n)_Z))0]= 0$.
\end{proof}

\begin{corollary} \label{cor:autom-defect}
  There exists a deterministic automaton
  $\mathcal{D} = ⟨ Q, \bar{B}, Γ, q_0, τ ⟩$ with an output function
  $τ : Q → \bar{B}$ such that $Γ(q_0,0) = q_0$ and
  $τ(Γ(q_0, (m)_Z ⊟ (n)_Z)) = δ(m-n,n)$ for all $m ⩾ n ⩾ 0$.
\end{corollary}
\begin{proof}
  It suffices to combine the three deterministic automata for
  $X_{\bar{1}}$, $X_0$ and~$X_1$ obtained in Proposition~\ref{pro:defect-reg}. 
\end{proof}
It should be noted that the set $\{ (m)_Z ⊟ (n)_Z : m ⩾ n ⩾ 0 \}$ is
strictly contained in~$\bar{B}^*$.  Indeed, it only contains words with no
consecutive occurrences of either the digit~$1$ or the digit~$\bar{1}$.
This gives us some freedom in the construction of an automaton~$\mathcal{D}$
satisfying the statement of the previous corollary, and it explains why
in~$\mathcal{D}$, one cannot follow all possible words.  We give in
Figure~\ref{fig:autom-defect} an automaton with $5$ states computing the
defect $δ(m-n,n)$; the value of the output function~$τ$ appears inside the
states.

\begin{figure}[htbp]
  \begin{center}
  \begin{tikzpicture}[initial text=,->,>=stealth',semithick,auto,inner sep=2pt]
  \tikzstyle{every state}=[minimum size=20]
  \node[state with output,initial] (q0) at (-2.5,2.5) {$q_0$ \nodepart{lower} $\tcr{0}$};
  \node[state with output] (q1) at (0,2.5) {$q_1$ \nodepart{lower} $\tcr{0}$};
  \node[state with output] (q2) at (2.5,2.5) {$q_2$ \nodepart{lower} $\tcr{0}$};
  \node[state with output] (q3) at (0,0) {$q_3$ \nodepart{lower} $\tcr{\bar{1}}$};
  \node[state with output] (q4) at (2.5,0) {$q_4$ \nodepart{lower} $\tcr{1}$};
  \path (q0) edge[out=120,in=60,loop] node {$0$} (q0);
  \path (q0) edge[bend left=10] node {$1$} (q1);
  \path (q1) edge[bend left=10] node {$\bar{1}$} (q3);
  \path (q1) edge[bend left=10] node {$0$} (q2);
  \path (q1) edge[out=120,in=60,loop] node {$1$} (q1);
  \path (q2) edge[out=120,in=60,loop] node {$\bar{1}$} (q2);
  \path (q2) edge[bend left=10] node {$0,1$} (q1);
  \path (q3) edge node {$\bar{1}$} (q2);
  \path (q3) edge[bend left=10] node {$0$} (q1);
  \path (q3) edge[bend left=10] node {$1$} (q4);
  \path (q4) edge[bend left=10] node {$\bar{1},0$} (q3);
  \path (q4) edge node {$1$} (q2);
  \end{tikzpicture}
  \end{center}
  \caption{An automaton computing $δ(m-n,n)$, given $(m)_Z ⊟ (n)_Z$ and $m
    ⩾ n ⩾ 0$.}
  \label{fig:autom-defect} 
\end{figure}

\subsection{Z-Mahler  equations and weighted automata}
\label{sec:operator}

The following definition of the operator~$Φ$ is exactly analogous to that
for base-$k$ defined in Section~\ref{sec:intro},
but it is based on the function~$ϕ$ rather than the function $n ↦ kn$.  Let
$R$ be a commutative ring.  The Z-Mahler operator $Φ: R⟦ x⟧ → R⟦ x⟧$ is
defined as follows.
\begin{equation} \label{eq:Phi-power}
  Φ\Bigl(∑_{n ⩾ 0}{f_nx^n}\Bigr) ≔ ∑_{n ⩾ 0}{f_nx^{ϕ(n)}}.
\end{equation}

Note that, unlike the $k$-Mahler case, we do not have $Φ(fg)=Φ(f)Φ(g)$ in 
general, as the function $ϕ$ is not linear.  Nevertheless, as $ϕ$ is linear
over integers with disjoint support, we have that $Φ(fg)=Φ(f)Φ(g)$ if for
each pair of non-zero coefficients $f_m≠ 0$ and $g_n≠ 0$, $m$ and $n$ have
disjoint support.

The equation 
\begin{align}\label{eq:Z-Mahler-equation}
  P(x,y) = ∑_{i=0}^dA_i(x)Φ^i(y) = 0,
\end{align}
with $A_i$ defined as in~\eqref{eq:coefficients}, is called a
\emph{Z-Mahler equation}, and if $f ∈ R⟦ x⟧$ satisfies
$∑_{i=0}^dA_i(x)Φ^i(f) = 0$, then it is called 
\emph{Z-Mahler}; we also
say that $f$ is a solution of the functional equation~$P=0$.  As in the case
of standard Mahler equations, $d$ is the \emph{exponent} of the equation
$P$, and the maximum degree $h$ of the polynomials $A_0(x), …, A_d(x)$ the
\emph{height} of~$P$.
Also, in analogy
to~\eqref{eq:condensed-recurrence-brut}, a solution $f(x)= ∑_{n ⩾ 0}{f_nx^n}$
to~\eqref{eq:Z-Mahler-equation} satisfies

\begin{align} \label{eq:condensed-recurrence-brut-Zeckendorf}
  f_n =    ∑_{ϕ^i(k)+j=n} α_{i,j}f_k \, .
\end{align}
and where here also we set $α_{i,j}=0$ for $i,j$ outside the bounds given
by the equation.

For example, the polynomial $f(x) = 1+x$ is the solution of the Z-Mahler
equation $(1+x^2)f(x) = (1+x)Φ(f(x))$ because $Φ(f(x)) = 1+x^2$, by
Definition~\eqref{eq:Phi-power} and Table~\eqref{fig:valuesofphi} The
following is a slightly less trivial example, inspired by
\cite[Proposition~1]{Becker-1994}.
 
\begin{example}\label{ex:number-of-representations}\normalfont
  We return to Example \ref{ex:preliminary}, whose automaton computes the
  number~$a_n$ of representations of~$n$ as a sum of distinct Fibonacci
  numbers.  Consider the series $f(x) ≔ ∑_{n⩾0}a_nx^n$; we also have
  $f(x) = \prod_{n ⩾ 0}(1+x^{F_n})$.  This series $f(x)$ is the solution of
  the equation $f(x) = (1+x)Φ(f(x))$.  It can be indeed verified that
  $Φ(∏_{n ⩾ 0}(1+z^{F_n})) = ∏_{n ⩾ 0}(1+z^{F_{n+1}}) = ∏_{n ⩾
    1}(1+z^{F_n})$ because the terms of the product have disjoint
  exponents.
\end{example}

It is not a coincidence that the solution in
Example~\ref{ex:number-of-representations} can be computed using a weighted
automaton, as we will see next.

\subsubsection{From isolating Z-Mahler equations to weighted automata}

As in the case of $k$-Mahler equations, a Z-Mahler equation $P(x,y)$ is
\emph{isolating} if $A_0 = 1$, i.e., $P(x,y) =y- ∑_{i=1}^dA_i(x)Φ^i(y)$.
The aim of this section is to show that any solution of an isolating
Z-Mahler equation is Z-regular, i.e., if $f(x) = ∑_{n ⩾ 0}f_nx^n$ is
the solution of~$P$ with initial condition~$f_0$, then there exists a
weighted automaton~$\mathcal{A}$ such that $f_n = \weight_{\mathcal{A}}((n)_Z)$ for each
integer $n ⩾ 0$.

The construction of the weighted automaton is similar to the one we gave in
Section~\ref{sec:equation-to-automaton} for $k$-Mahler equations, but it is
more involved due to the non-linearity of the function~$ϕ$.  

Now we describe the weighted automaton~$\mathcal{A}$ computing the solution of an
isolating Z-Mahler equation.  Recall that $B = \{0, 1\}$, $\bar{B} =
\{\bar{1}, 0, 1\}$ and $C=\{\bar{1}, 0, 1, 2\}$.  Let $\mathcal{D} = 
⟨ Q, \bar{B}, Γ,\{q_0\},τ ⟩$ be a deterministic automaton given by
Corollary~\ref{cor:autom-defect}, with $τ : Q → \bar{B}$ such that
$τ(Γ(q_0, (m)_Z ⊟ (n)_Z)) = δ(m-n,n)$ for all integers $m ⩾ n ⩾ 0$; it can
be assumed that $Γ(q_0,0) = q_0$.

We now come to the definition of the weighted automaton associated to a
Z-Mahler equation of height $h$ and exponent $d$. Let
$P(x,y) = ∑_{i=0}^dA_i(x)Φ^i(y) = 0$ be an isolating Z-Mahler equation with
$A_i(x) = ∑_{j=0}^h α_{i,j}x^j$ for $1 ⩽ i ⩽ d$.  Set
\begin{displaymath}
    \tilde{h} ≔ \left⌊\frac{h+3-φ}{φ-1}\right⌋
    \quad\text{and}\quad
    g ≔ |(\tilde{h})_Z|.
\end{displaymath}
We define the state set~$S$ as 
\begin{displaymath}
  S ≔ \{ s_{i,j,q,u} : 0 ⩽ i ⩽ d, 0 ⩽ j ⩽ \tilde{h},
                                q ∈ Q \text{ and } u ∈ B^g  \}.
\end{displaymath}
 Define the function $\hat{δ}: S → \bar{B}$  by
\begin{displaymath}
  \hat{δ}(s_{i,j,q,u}) = τ(Γ(q, u ⊟ (j)_Z)).
\end{displaymath}
Note that if $q\neq q_0$, the automaton in Figure~\ref{fig:autom-defect}
can be fed with any word in $\bar{B}^*$, i.e., $τ(Γ(q, u ⊟ (j)_Z))$ is
well defined.  If $q=q_0$, we will see in the second statement of
Lemma~\ref{lem:delta-hat-well-defined} that then we will only be
concerned with states $s_{i,j,q_0,u}$ where $[u]_Z ⩾ (j)_Z$ in which case
$\hat{δ}$ is also well defined. Define the transition set~$Δ$ as follows.
\begin{align*}
  Δ  ≔ {} & \left\{ s_{i,j,q,au} \trans{b:1}
                     s_{i+1,ℓ,Γ(q,a),ub} :
                     \begin{array}{c}
                       0 ⩽ i ⩽ d-1, \\
                       0 ⩽ j ⩽ \tilde{h} \\
                       0 ⩽ ℓ = ϕ(j)+\hat{δ}(s_{i,j,q,au})+b ⩽ \tilde{h} \\
                       a,b ∈  B, u ∈ B^{g-1}
                     \end{array}
            \right\} \\
   {} ∪ {} & \left\{ s_{i,j,q,au} \trans{b:α_{i+1,ℓ-k}}
                     s_{0,k,Γ(q,a),ub} :
                     \begin{array}{c}
                       0 ⩽ i ⩽ d-1 \\
                       0 ⩽ j,k ⩽ \tilde{h} \\
                       0 ⩽ ℓ = ϕ(j)+\hat{δ}(s_{i,j,q,au})+b \\
                       0 ⩽ ℓ-k⩽ h \\
                       a,b ∈  B, u ∈ B^{g-1} 
                     \end{array}
              \right\}.
\end{align*}
We set the initial and final weights $I$ and $F$ as
\begin{align*}
  I(s_{i,j,q,u}) & ≔
  \begin{cases}
    f_0 & \text{if $j = 0$, $q = q_0$ and $u = 0^g$,} \\
    0   & \text{otherwise.}
  \end{cases} \\
  F(s_{i,j,q,u}) & ≔
  \begin{cases}
    1 & \text{if $i = 0$ and $j = 0$,} \\
    0 & \text{otherwise.}
  \end{cases}
\end{align*}

We call the automaton $\mathcal{A}≔ ⟨S,A,Δ,I,F⟩$ the \emph{weighted automaton
  associated to $P$ and~$f_0$}, $\mathcal{A}= \mathcal{A}(P,f_0)$.

The first two components $i$ and~$j$ of a state~$s_{i,j,q,u}$
in~$\mathcal{A}$ play the same role as in the base-$k$ case. We will show
how the last two components are needed to track the linearity defect.  The
following lemma describes the evolution of the third and fourth components
of the states along a path.

\begin{lemma} \label{lem:movement}
  Let $s_{i',j',q',u'} \trans{w:r} s_{i,j,q,u}$ be a path in the
  automaton~$\mathcal{A}(P,f_0)$. Then $u$ is the suffix of length~$g$
  of~$u'w$ and $q = Γ(q',v)$ where $vu = u'w$, that is, $v$ is the prefix
  of length~$|w|$ of~$u'w$.
\end{lemma}
\begin{proof}
  The proof is a straightforward induction on the length of~$w$.  The base
  case $|w| = 0$ is trivial and the induction step follows directly from
  the definition of the transition set~$Δ$.
\end{proof}

Rephrasing Lemma~\ref{lem:movement} when $q'$ is the initial state~$q_0$
of~$\mathcal{D}$ and $u' = 0^g$ gives Corollary~\ref{cor:movement}. In particular, it
tells us that if $s_{i,j,q,u}$ is accessible from a state with non-zero
initial weight, then,  since these states  have $0^g$ as their last component, the word~$u$ must be a suffix of~$0^g w$.  This implies that $u$ cannot have consecutive occurrences of the digit~$1$.
We have not taken this into account, in the definition of~$S$, in order to
simplify the definition of~$\mathcal{A}$. However this fact is used to bound the
number of states of~$\mathcal{A}$ in Lemma~\ref{lem:state-count}.
\begin{corollary} \label{cor:movement}
  Let $s_{i',j',q_0,0^g} \trans{w:r} s_{i,j,q,u}$ be a path in the
  automaton~$\mathcal{A}(P,f_0)$. Then $u$ is the suffix of length~$g$ of~$0^gw$
  and $q = Γ(q_0,v)$ where $vu = 0^gw$, that is, $v$ is the prefix
  of length~$|w|$ of~$0^gw$.
\end{corollary}

The following lemma states that the function~$\hat{δ}$ indeed tracks the
linearity defect, starting from the initial state on input of any word $w$,
as a function of the state where it arrives. Note that the proof of this
lemma also justifies the definition of $g$.

\begin{lemma}\label{lem:delta-hat-well-defined} 
  Let $s_{i',0,q_0,0^g} \trans{w:r} s_{i,j,q,u}$ be a path in the
  automaton~$\mathcal{A}(P,f_0)$.
  \begin{itemize}\itemsep0cm
  \item[a)] There exists $k ⩾ 0$ such that $ϕ^i(k) + j = [w]_Z$ and thus
    $j ⩽ [w]_Z$.
  \item[b)] The function $\hat{δ}(s_{i,j,q,u})$ is well-defined and
    $\hat{δ}(s_{i,j,q,u}) = δ([w]_Z-j,j)$.
  \end{itemize}
\end{lemma}
\begin{proof}
  The two statements are simultaneously proved by induction on the length
  of the word~$w$.  If the word~$w$ is empty,  then  $j = j'=0$, we can take $k=0$, and 
  $\hat{δ}(s_{i,j,q,u}) = \hat{δ}(s_{i',0,q_0,0^g}) = τ(Γ(q_0, 0^g ⊟ (0)_Z)) =δ(0,0) $ by Corollary~\ref{cor:autom-defect}.

  Now we suppose that $w = w'b$ where $b$ is the last digit of~$w$.  The
  path $s_{i',0,q_0,0^g} \trans{w:r} s_{i,j,q,u}$ can then be decomposed
  \begin{displaymath}
    s_{i',0,q_0,0^g}  \trans{w':r'} s_{i'',j'',q',u'} \trans{b:r''} s_{i,j,q,u}.
  \end{displaymath}
  For the first statement, we distinguish two cases depending on the form
  of the last transition $s_{i'',j'',q',u'} \trans{b:r''} s_{i,j,q,u}$.
  
  We first suppose that this last transition of the path is a transition of
  the form $s_{i'',j'',q',u'} \trans{b:1} s_{i,j,q,u}$ where $i = i''+1$ and
  $j = ϕ(j'') + \hat{δ}(s_{i'',j'',q',u'})+b$.  By the induction hypothesis,
  $\hat{δ}(s_{i'',j'',q',u'}) = δ([w']_Z-j'',j'')$ and there exists an integer
  $k ⩾ 0$ such that $ϕ^{i''}(k) + j'' = [w']_Z$.  Applying $ϕ$ to the equality
  $ϕ^{i''}(k) = [w']_Z -j''$ yields $ϕ^i(k) = ϕ([w']_Z-j'') =
  ϕ([w']_Z)-ϕ(j'')-δ([w']_Z-j'',j'') = ϕ([w']_Z)+b-j = [w]_Z-j$.
  
  Next, we suppose that this last transition of the path is a transition of
  the form $s_{i'',j'',q',u'} \trans{b:α_{i''+1,ℓ-j}} s_{0,j,q,u}$ where
  $ℓ = ϕ(j'') + \hat{δ}(s_{i'',j'',q',u'})+b$.  By definition of the
  transitions, $0 ⩽ ℓ-j$, and we have $j ⩽ ℓ$ and, by the inductive
  hypothesis, \[ℓ = ϕ(j'') + \hat{δ}(s_{i'',j'',q',u'})+b =
  ϕ(j'') + δ([w']_W-j'',j''))+b = ϕ([w']_Z)-ϕ([w']_Z-j'')+b,\] and thus
  $j ⩽ ℓ ⩽ ϕ([w']_Z)+b = [w]_Z$.  Since $i = 0$, $ϕ^i(k) + j = [w]_Z$ where
  $k = [w]_Z-j$.  This completes the proof of the first statement.

  The only missing transition in the automaton pictured in
  Figure~\ref{fig:autom-defect} is the transition with label~$\bar{1}$
  leaving state~$q_0$. If the state $q$ in $s_{i,j,q,u}$ is equal to~$q_0$,
  then $[u]_Z = [w]_Z$, since if the state is still $q_0$, only $0$ has
  been fed to the automaton.  This means that the word~$v$ in the statement
  of Corollary~\ref{cor:movement} is a block of zeros.  and therefore
  $j ⩽ [u]_Z$ and the most significant digit of $u ⊟ (j)_Z$ must be~$1$.
  This shows that $\hat{δ}(s_{i,j,q,u}) = τ(Γ(q, u ⊟ (j)_Z))$ is always
  well-defined, since in the automaton $\mathcal{D}$, once one leaves $q_0$,
  there is no return.
  
  Let $m$ be the length of~$w = w'b$.  By Corollary~\ref{cor:movement}, one
  has $q = Γ(q_0,v)$ where $vu = 0^gw$ and $|v| = m$.  This can be
  rewritten $q = Γ(q_0,v ⊟ 0^{m})$.  Recall that the function~$\hat{δ}$ is
  defined by $\hat{δ}(s_{i,j,q,u}) = τ(Γ(q, u ⊟ (j)_Z))$.  Since by
  definition of the states, $j⩽ \tilde{h}$, so $|(j)_Z| ⩽ g$, and $|u|=g$,
  then by definition of $\mathcal{D}$,
  $Γ(q, u ⊟ (j)_Z) = Γ(q_0,vu ⊟ 0^m 0^{g-|(j)_Z|}(j)_Z) = Γ(q_0, w ⊟
  0^{m-|(j)_Z| }(j)_Z)$.  It follows that
  $\hat{δ}(s_{i,j,q,u}) = δ([w]-j,j)$.
\end{proof}

The following lemma is used in the proof of Proposition~\ref{pro:key-Z} below.
\begin{lemma}\label{lem:essentially-one-path}
  Let $s_{i'',0,q_0,0^g} \trans{w':r'} s_{i',j',q',au} \trans{b:r}
  s_{i,j,q,ub}$ be a path in the automaton~$\mathcal{A}(P,f_0)$ where $w'$
  is a word and $b$ is a digit. If $i ⩾ 1$ then the state $s_{i',j',q',au}$
  is unique and does not depend on $i''$.
\end{lemma}
\begin{proof}
  By Corollary~\ref{cor:movement}, $au$ must be the suffix of length~$g$ of
  $0^gw'$ and the state~$q'$ is given by $q' = Γ(q_0, v)$ where $vau =
  0^gw'$.  Since $i ⩾ 1$, the last transition must be  of the form
  $s_{i',j',q',au} \trans{b:1} s_{i,j,q,ub}$ where $i = i'+1$ and $j =
  ϕ(j') + \hat{δ}(s_{i',j',q',u'})+b$.  This shows that $i' = i-1$ and that
  $r=1$.  It remains to show that $j$ is also unique.  By
  Lemma~\ref{lem:delta-hat-well-defined}, $\hat{δ}(s_{i',j',q',u'})$ is
  equal to $δ([w']_Z-j',j')$ and the equality $j = ϕ(j') +
  \hat{δ}(s_{i',j',q',u'})+b$ can be rewritten $ϕ([w']_Z-j') = ϕ([w']_Z )+b-j$.
  Since the function~$ϕ$ is one-to-one, there is at most one integer~$j'$
  satisfying this equality. Note that this proof does not depend on $i''$,
  so the statement follows. 
\end{proof}

We bring this together now to obtain the following version of
Proposition~\ref{prop:key}.
\begin{proposition} \label{pro:key-Z}
  Let $R$ be a commutative ring, let $P(x,y)∈ R[x,y]$ be an isolating
  Z-Mahler equation of exponent~$d$ and height~$h$, and let $f = ∑_{n ⩾
    0}f_n x^n$ be a solution of $P=0$.  If $\mathcal{A}(P,f_0)$ is the weighted
  automaton associated to $P(x,y)$ and $f_0 ∈ R$, and if $w ∈ \{0,1\}^*$
  has no consecutive occurrences of the digit~$1$, then
  \begin{displaymath}
    \weight_{\mathcal{A},s_{i,j, q, u}}^{*}(w) =
    \begin{cases}
      f_{k} & \text{  whenever } 
        \begin{cases} 
          ϕ^i(k)+j=[w]_Z ,\\
          q = Γ(q_0,w), \text{ and }  \\
          u \text{ is the suffix of } 0^gw
      \end{cases}  \\
      0      & \text{otherwise.}
    \end{cases}
  \end{displaymath}
\end{proposition}
Note that $i$, $j$ and~$w$ being given, there exists at most one
integer~$k$ that satisfies the equation $ϕ^i(k)+j=[w]_Z$, as the
function~$ϕ$ is one-to-one.  This means that the integer~$k$ implicitly
given by $ϕ^i(k)+j=[w]_Z$ is well-defined.

\begin{proof}
  The proof is by induction on the length of the word~$w$.  If $w$ is the
  empty word~$ε$, then $\weight_{\mathcal{A},s_{i,j,q,u}}^*(ε) = I(s_{i,j,q,u})$;
  this equals~$f_0$ if $j= 0$, $q=q_0$ and $u=0^g$, and equals~$0$
  otherwise, so the claim is proved for $w = ε$ since $[ε]_Z = 0$.
  
  For $|w| > 0$, we consider first the easier case of a state
  $s_{i,j, q, ub}$ with $i ⩾ 1 $. We are only concerned with states, and
  transitions into these states, that are part of a path starting at an
  initial state.  By Lemma~\ref{lem:essentially-one-path}, there is a
  unique such transition
  $s_{i_0,0,q_0,0^g} \trans{w:r}s_{i-1,j',q',au} \trans{b:1} s_{i,j,q,ub}
  $. Here
  $j = ϕ(j')+\hat{δ}(s_{i-1,j',q',au})+b = ϕ(j') + δ([w]_Z-j',j')+b$, by
  Lemma~\ref{lem:delta-hat-well-defined}. Then by the inductive hypothesis,
  \begin{displaymath}
    \weight_{\mathcal{A},s_{i,j, q, ub}}^{*}(wb)  =
    \weight_{\mathcal{A},s_{i-1,j', q', au}}^{*}(w) = f_k,
  \end{displaymath}
  where $\phi^{i-1}(k)+j'=[w]_Z$. Applying $ϕ$ to this last equality, we obtain
  \begin{align*} 
    \phi^{i}(k) & = \phi ([w]_Z - j') \\ 
                & = \phi ([w]_Z) -\phi( j') - \delta([w]_Z-j',j')\\ 
                & = [wb]_Z -b-\phi( j') - \delta([w]_Z-j',j')\\
                & = [wb]_Z -j,
  \end{align*}
  from which the statement of the proposition follows for the case $i>0$.
   
  Now we consider the case case $i = 0$, so that our state is $s_{0,j, q,
    ub}$.  As the transitions that concern us are of the form
  $s_{i',j',q',au} \trans{b:α_{i'+1,ℓ'-j}} s_{0,j,q,ub}$, henceforth we
  will implicitly only sum over states $s_{i',j',q',au}$ which satisfy
  this.  Also, as we are only concerned with paths that commence at an
  initial state, we have by Lemma~\ref{lem:delta-hat-well-defined} that
  $ℓ' = \phi(j')+δ([w]_Z-j',j')+b$.  Thus
  \begin{align*}
    \weight_{\mathcal{A},s_{0,j,q,ub}}^*(wb) & =
      ∑_{i',j',q',a}α_{i'+1,ℓ'-j}\weight_{\mathcal{A},s_{i',j',q',au}}^*(w) \\
          & = ∑_{ϕ^{i'}(k)+j'=[w]_Z}α_{i'+1,ℓ'-j}f_k
  \end{align*}
  where to obtain the last equality we have applied the inductive hypothesis.
  Since $ϕ$ is one-to-one, the equality $ϕ^{i'}(k)=[w]_Z-j'$ is equivalent
  to the equality 
  \begin{align*}ϕ^{i'+1}(k)=ϕ([w]_Z-j') &= ϕ([w]_Z)-ϕ(j')-δ([w]_Z-j',j')\\
  & = ϕ([w]_Z)-ℓ'+b.
  \end{align*}
     Setting $ℓ = ℓ'-j$, we obtain $ϕ^{i'+1}(k)+ℓ = [wb]_Z-j$ and thus
  \begin{align*}
    \weight_{\mathcal{A},s_{0,j,q,ub}}^*(wb) & =
      ∑_{ϕ^{i'+1}(k)+ℓ = [wb]_Z-j}α_{i'+1,ℓ}f_k \\
      & = ∑_{ϕ^i(k)+ℓ = [wb]_Z-j}α_{i,ℓ}f_k \\
      & = f_{[wb]_Z-j}
  \end{align*}
  where in the penultimate line we set $i=i'+1$, and where we used
  Equation~\eqref{eq:condensed-recurrence-brut-Zeckendorf} to get to the last
  line.
\end{proof}

The following theorem states that the automaton~$\mathcal{A}(P,f_0)$
defined above indeed computes the unique solution of the equation
$P(x,f(x))$ satisfying $f(0) = f_0$.
\begin{theorem} \label{thm:main-Z} Let $R$ be a commutative ring, and let
  $P(x,y) ∈ R[x,y]$ be an isolating Z-Mahler equation of exponent~$d$ and
  height~$h$.  Then the automaton $\mathcal{A} = \mathcal{A}(P,f_0)$
  associated to $P$ and~$f_0$ satisfies
  \begin{displaymath}
    \weight_{\mathcal{A}}(w) = f_{[w]_Z}
  \end{displaymath}
  for any $w$, where $f = ∑_{n ⩾ 0}f_nx^n$ is a solution of $P(x,f(x))=0$.
  Consequently if $R$ is finite, then there exists a deterministic
  automaton, and a constant~$C$, depending only on~$φ$, with at most
  $|R|^{Cdh^2}$ states that generates~$f(x)$.
\end{theorem}

In the classical $q$-numeration, with $q$ a power of a prime, if $P$ is an
isolating Ore polynomial over~$\mathbb{F}_q$ of degree~$q^d$ and
height~$h$, then we have seen that a weighted automaton generating a
solution of~$P$ will have at most $dh$ states.  This should be compared to
the bound above, where the exponent $Cdh^2= 320dh^2$ has an extra factor of
$320h$, with $320$ being a function of~$φ$ (see
Lemma~\ref{lem:state-count}).  This extra factor arises because we need to
carry extra information, in the form of a word of length~$g$, which is used
to compute the linearity defect.

\begin{proof}
  The states with a non-zero final weight are the states of the form
  $s_{0,0,q,u}$ whose final weight is given by $F(s_{0,0,q,u}) = 1$.  For
  each input word~$w$, that is, one with no consecutive occurrences of the
  digit~$1$, then by Proposition~\ref{pro:key-Z}, there exists exactly one
  state $s_{0,0,q,u}$ where $\weight_{\mathcal{A},s_{0,0, q, u}}^{*}(w)$ is
  non-zero: This unique state is the state $s_{0,0, q, u}$ where $u$ is the
  suffix of length~$g$ of~$0^gw$ and $q$ is given by $q = Γ(q_0,v)$ where
  $v$ is the prefix of length~$|w|$ of $0^gw$ by Lemma~\ref{lem:movement}.
  Therefore $\weight_{\mathcal{A}}(w)$ is equal to
  $\weight_{\mathcal{A},s_{0,0, q, u}}^{*}(w)$ where $u$ and~$q$ satisfy
  the required properties, and
  $\weight_{\mathcal{A}}(w)= \weight_{\mathcal{A},s_{0,0, q,
      u}}^{*}(w)=f_{[w]_Z}$, again by Proposition~\ref{pro:key-Z}, since
  $i=j=0$. The bound on the number of states follows from Lemmas
  \ref{lem:truncateZ} and~\ref{lem:state-count} below.
\end{proof}

The following lemma is the analog of Lemma~\ref{lem:truncate}. It justifies
the choice $\tilde{h} = \frac{h+3-φ}{φ-1}$, as this is an upper bound on
indices for states that are the range of paths of positive weight.
\begin{lemma} \label{lem:truncateZ}
  Let $h$ be the height of the given $Z$-Mahler equation. If
  $j ⩾ \frac{h+3-φ}{φ-1} = hφ+3φ-φ^2$ then there are no paths with positive
  weight in the associated automaton from~$s_{i,j,q,u}$ to~$s_{0,0,q'u'}$.
\end{lemma}
\begin{proof}
  Let $s_{i,j,q,u} \trans{b:α} s_{i',j',q',u'}$ be a transition in the
  weighted automaton~$\mathcal{A}$ associated to $(α_{i,j})_{i⩾ 1, j⩾ 0}$
  where $α_{i,j}=0$ if $j ⩾ h+1$.  We claim that if $α ≠ 0$ and
  $j ⩾ \frac{h+3-φ}{φ-1}$ then $j' ⩾ \frac{h+3-φ}{φ-1}$.  If
  $s_{i,j,q,u} \trans{b:α} s_{i',j',q',u'}$ is a transition of the form
  $s_{i,j,q,u} \trans{b:1} s_{i+1,j',q',u'}$ then by definition,
  $j' = ϕ(j) + \hat{δ}(s_{i,j,q,u}) + b$.  If $j = 0$, then
  $\hat{δ}(s_{i,j,q,u})= \hat{δ}(s_{i,0,q,u}) = τ(Γ(q, u ⊟ (0)_Z))=τ(Γ(q, u
  ))$, and now, since $u\in B^g$, inspection of
  Figure~\ref{fig:autom-defect} yields $\hat{δ}(s_{i,j,q,u}) = 0$, so
  $j'⩾ j$. If $j > 0$, then $ϕ(j) ⩾ j+1$.  If
  $s_{i,j,q,u} \trans{b:α} s_{i',j',q',u'}$ is a transition of the form
  $s_{i,j,q,u} \trans{b:α_{i+1,ℓ-j'}} s_{0,j',q',u'}$ so that
  $ℓ = ϕ(j) + \hat{δ}(s_{i,j,q,u}) + b$, then, since $α_{i+1,ℓ-j'}$ is
  assumed non-zero, we have by definition of the transitions,
  $ϕ(j) + \hat{δ}(s_{i,j,q,u}) + b-j'⩽ h$, and
  \begin{displaymath}
    j'⩾ ϕ(j)-1-h ⩾ jφ+φ-3-h ⩾ \frac{h+3-φ}{φ-1},
  \end{displaymath}
  where the first inequality follows because $\hat{δ}(s_{i,j,q,u})⩾ -1$ and
  $b⩾ 0$, the second inequality follows from Lemma~\ref{lem:floor-Cartier},
  and the third inequality follows from the bound on $j$ and the fact that
  $φ^{2}-φ-1=0$.  The statement of the lemma follows.
\end{proof}

The following lemma provides an upper bound for the number of states in the
automaton given by Theorem~\ref{thm:main-Z}. Note that we do not claim that
this automaton is optimally small. The series of
Example~\ref{ex:number-of-representations} is generated by an automaton
with $3$ states (see Example~\ref{ex:preliminary}). However, as this series
is the solution of a Mahler equation of height~$1$ and exponent~$1$, then
$\tilde{h} =3=g$, and inspecting the proof of Lemma~\ref{lem:state-count},
the automaton defined in Theorem~\ref{thm:main-Z} would have $100$ states.

\begin{lemma} \label{lem:state-count}
  Let $\mathcal{A}$ be the weighted automaton associated to a Z-Mahler
  equation of height~$h⩾ 1$ and exponent~$d$.  Then the number of
  states~$|S|$ of~$\mathcal{A}$ is bounded by $45 d(3h+1) h$ and
  asymptotically, $|S| ⩽ 15 dφ h(φ h+1)(1+o(1)) $ as $h$
  tends to~$∞$.
\end{lemma}
We do not consider the extreme case $h = 0$ because each solution of a
Z-Mahler equation of height~$0$ must be a constant series.  If
$f = ∑_{n ⩾ 0}f_nx^n$ is a non-constant solution of such an equation, the
integer $\min \{ n ⩾ 1 : f_n ≠ 0\}$ must satisfy $n = ϕ^k(n)$ for some
positive integer~$k$ and this is not possible.
\begin{proof}
  Elements of $S$ are indexed by $(i,j,q,u)$, where $0⩽ i ⩽ d$,
  $0⩽ j⩽ \tilde{h}$, $q\in Q$, and $u\in B^g$, but in fact we are only
  interested in words $u$ where there are no two consecutive occurrences of
  the digit~$1$, and also, the states where $i=d$ do not have any outgoing
  edges with positive weights. Recall that $|Q|=5$, and the number of words
  of length~$g$ with no consecutive occurrences of~$1$ is equal to the
  Fibonacci number~$F_{g+1}$.  Thus $|S| ⩽ 5d(\tilde{h}+1)F_{g+1}$.
  Now~$\tilde{h}=⌊\frac{h+3-φ}{φ-1}⌋⩽ 3h$, so $|S| ⩽ 5d(3h+1)F_{g+1}$.
  Note that $F_{g-1} ⩽ \tilde{h}$ by definition of~$g$. Since
  $F_{n+2} ⩽ 3F_n$ for each $n ⩾ 0$, this implies that
  $F_{g+1}⩽ 3\tilde{h}$. Thus $|S| ⩽ 15 d(3h+1)\tilde{h} ⩽ 45 d(3h+1) h$.
    
  To obtain the asymptotic bound, as $h → ∞$, we have
  $\tilde{h} = φh(1+o(1))$. Thus
  $|S| ⩽ 15d(\tilde{h}+1)\tilde{h} = 15dφ h(φ h +1 )(1+o(1))$.
\end{proof}

\subsubsection{From weighted automata to Z-Mahler equations}
\label{sec:automata-to-Z-Mahler}

This section is almost standard, but we include it for completeness.  We
first redo Example~\ref{ex:automaton-to-equation}

\begin{example}
  We consider again the automaton of Figure~\ref{fig:weighted0}, associated
  to Examples \ref{ex:reg-not-aut} and~\ref{ex:automaton-to-equation},
  except that here we use it to generate the term
  $a_n ≔ \weight_{\mathcal{A}}((n)_Z)$.  The weights of this automaton are
  in the field~$\mathbb{F}_2$ or the ring~$ℕ$, but the computation of the
  equation satisfied by $t(x) = ∑_{n ⩾ 0} a_n x^n$ can be carried out in
  any ring~$R$.  Let $t(x) = ∑_{n ⩾ 0} a_n x^n$, and $s = s(x)$ in~$R⟦ x⟧$
  be defined similarly as in Example~\ref{ex:automaton-to-equation}. Here
  we partition $ℕ = ϕ(ℕ) \uplus (ϕ^2(ℕ)+1)$.  We have, from
  Figure~\ref{fig:weighted0},
  \begin{alignat}{2} \label{eq:recursive-Z}
    a_{ϕ(n)} &= a_n & \quad\text{and}\quad a_{ϕ^2(n)+1} &= a_n ⊕ b_n \\
    b_{ϕ(n)} &= b_n & \quad\text{and}\quad b_{ϕ^2(n)+1} &= b_n.\notag
  \end{alignat}
  where the symbol~$⊕$ denotes the sum in~$R$.  For brevity, we write
  $t=t(x)$ and $s=s(x)$. Then Equations~\eqref{eq:recursive-Z} give
  \begin{equation} \label{eq:recursive-2-F}
    t = Φ(t) + xΦ^2(t) + xΦ^2(s)
    \quad\text{ and }\quad
    s = Φ(s) + xΦ^2(s).
  \end{equation}
  Applying $Φ$ to the equations~\eqref{eq:recursive-2-F}, we have
  \begin{align} \label{eq:recursive-3-F}
    Φ(t) & = Φ^2(t) + Φ(xΦ^2(t))+Φ(xΦ^2(s)) \nonumber \\
         & = Φ^2(t) + x^2Φ^3(t)+x^2Φ^3(s),
  \end{align}
  as $x$ and $Φ^2(t)$ or $Φ^2(s)$ have disjoint support, and similarly,
  \begin{align}
    Φ(s)  & = Φ^2(s) + x^2 Φ^3(s),\nonumber\\
    Φ^2(t) & =Φ^3(t) + x^3 Φ^4(t)+x^3Φ^4(s), \text{ and } \label{eq:a0}\\
    Φ^2(s) & = Φ^3(s) + x^3 Φ^4(s). \nonumber
  \end{align}
  Using these in~\eqref{eq:recursive-3-F}, we obtain 
  \begin{align}
    t    & = (1+x+x^2)Φ^3(t)+ (x^3+x^4)Φ^4(t)+(x+x^2)Φ^3(s)+ (x^3+2x^4)Φ^4(s),
             \text{and} \label{eq:a1}\\
    Φ(t) & = (1+x^2)Φ^3(t)+ x^3Φ^4(t)+x^2Φ^3(s)+ x^3 Φ^4(s) \label{eq:a2}
  \end{align}
  and, substituting \eqref{eq:a1}, \eqref{eq:a2} and~\eqref{eq:a0} for $t$,
  $Φ(t)$ and~$Φ^2(t)$ respectively in the following, we obtain
  \begin{displaymath}
    xt-(1+x)Φ(t)+(1-2x^2)Φ^2(t)+2x^2Φ^3(t)+x^5Φ^4(t) = 0 ,
  \end{displaymath}
  i.e., the Z-Thue-Morse power series is a solution of the isolating Mahler
  Z-equation $P(x,y)= xy-(1+x)Φ(y) +(1-2x^2)Φ^2(y)+2x^2Φ^3(y) +x^5 Φ^4(y)$.
\end{example}

In the following lemma, we show that if we have a finite family of series,
each of which can be written as a linear expression of the image under $Φ$
and~$Φ^2$ of the others, then this can be extended under application of
iterates of~$Φ$. We will later apply it to the collection of series
obtained from the same automaton, but with different final weights.
\begin{lemma} \label{lem:finite-kernel}
  Let $F = \{ s_1, …, s_m \}$ be a family of formal power series such that
  there exist two families of coefficients $(α_{i,j})_{1 ⩽ i,j ⩽ m}$ and
  $(β_{i,j})_{1 ⩽ i,j ⩽ m}$ satisfying for each $1 ⩽ i ⩽ m$,
  \begin{displaymath}
    s_i = ∑_{j=1}^m α_{i,j}Φ(s_j) + x∑_{j=1}^m β_{i,j}Φ^2(s_j).
  \end{displaymath}
  Then for each pair of integers $k$, $n$ with $0 ⩽ k ⩽ n$, there exist two
  families of polynomials $(p_{i,j}^{k,n})_{1 ⩽ i,j ⩽ m}$ and
  $(q_{i,j}^{k,n})_{1 ⩽ i,j ⩽ m}$ such that for each $1 ⩽ i ⩽ m$,
  \begin{displaymath}
   Φ^k(s_i) =  ∑_{j = 1}^m p_{i,j}^{k,n}Φ^n(s_j) +
                ∑_{j = 1}^m q_{i,j}^{k,n} Φ^{n+1}(s_j).
  \end{displaymath}
\end{lemma}
\begin{proof}
  Note that the hypothesis is the case $k = 0$ and $n =1$ of the
  statement.  The proof is by induction on the difference $n-k$.  The case
  $n = k$ is trivial and gives
  \begin{displaymath}
    p_{i,j}^{k,k} ≔
    \begin{cases}
      1 & \text{if $i = j$} \\
      0 & \text{otherwise} 
    \end{cases}
    \qquad\text{and}\qquad
    q_{i,j}^{k,k} ≔ 0.
  \end{displaymath}
  For $n = k+1$, we start from the hypothesis
  \begin{displaymath}
    s_i = ∑_{j=1}^m α_{i,j}Φ(s_j) + x∑_{j=1}^m β_{i,j}Φ^2(s_j)
  \end{displaymath}
  to get, by applying $Φ^k$ to both members,
  \begin{displaymath}
    Φ^k(s_i) = ∑_{j=1}^m α_{i,j}Φ^{k+1}(s_j) +
               Φ^k(x)∑_{j=1}^m β_{i,j}Φ^{k+2}(s_j),
  \end{displaymath}
  where we note that the polynomial~$x$ and the series $Φ^2(s_j)$ have
  disjoint support and that
  $Φ^k(x Φ^2(s_j)) = x^{ϕ^k(1)} Φ^{k+2} (s_j) = x^{F_k} Φ^{k+2} (s_j)$.  It
  follows that
  \begin{displaymath}
    p_{i,j}^{k,k+1} ≔ α_{i,j}
    \qquad\text{and}\qquad
    q_{i,j}^{k,k+1} ≔ β_{i,j}x^{F_k}.
  \end{displaymath}

  Now we suppose $k+2⩽ n$, and assume that the induction hypothesis holds
  for $n-(k+1)$ and $n-(k+2)$; we show the required statement holds for
  $n-k$.  From the induction hypothesis applied to $Φ^{k+1}(s_ℓ)$ and
  $Φ^{k+2}(s_ℓ)$, we have the following equalities:
  \begin{align*}
   Φ^{k+1}(s_ℓ) & = ∑_{j = 1}^m p_{ℓ,j}^{k+1,n}  Φ^n(s_j) +
                  ∑_{j = 1}^m q_{ℓ,j}^{k+1,n}  Φ^{n+1}(s_j) \\
   Φ^{k+2}(s_ℓ) & = ∑_{j = 1}^m  p_{ℓ,j}^{k+2,n} Φ^n(s_j)+
                  ∑_{j = 1}^m q_{ℓ,j}^{k+2,n} Φ^{n+1}(s_j)
  \end{align*}
  Combining these two equalities with the equality
  \begin{displaymath}
    Φ^k(s_i) = ∑_{ℓ=1}^m α_{i,ℓ}Φ^{k+1}(s_ℓ) +
               x^{F_k}  ∑_{ℓ=1}^m β_{i,ℓ} Φ^{k+2} (s_ℓ) 
  \end{displaymath}
  we get the following equalities, defining the required polynomials
  $p_{i,j}^{k,n}$ and~$q_{i,j}^{k,n}$ by
  \begin{align*}
    p_{i,j}^{k,n} & ≔ ∑_{ℓ=1}^m α_{i,ℓ}p_{ℓ,j}^{k+1,n} +
                    x^{F_k} ∑_{ℓ=1}^m β_{i,ℓ}p_{ℓ,j}^{k+2,n} \\
    q_{i,j}^{k,n} & ≔ ∑_{ℓ=1}^m α_{i,ℓ}q_{ℓ,j}^{k+1,n} +
                    x^{F_k}∑_{ℓ=1}^m β_{i,ℓ}q_{ℓ,j}^{k+2,n}.
  \end{align*}
\end{proof}

The following corollary follows directly from Lemma~\ref{lem:finite-kernel}.
\begin{corollary}\label{cor:regular-is-mahler}
  A Z-regular series is the solution of a Z-Mahler equation.
\end{corollary}
\begin{proof}
  Suppose that the series $f = ∑_{n ⩾ 0}f_nx^n$ is computed by the weighted
  automaton~$\mathcal{A}$, that is $f_n = \weight_{\mathcal{A}}((n)_Z)$ for
  each $n ⩾ 0$.  Let $\tuple{I,μ,G}$ be a matrix representation of
  dimension~$m$ of the weighted automaton~$\mathcal{A}$ as given in
  Section~\ref{sec:matrix-weighted}.  For $1 ⩽ i ⩽ m$, let $s_i$ be the
  series computed by the weighted automaton whose matrix representation is
  $\tuple{I,μ, G^{(i)}}$ where $G^{(i)}$ is the vector having $1$ in its
  $i$-th coordinate and $0$ in all other coordinates. Note that $f$ is
  equal to $f = ∑_{i=1}^m G_is_i$ where $G_i$ is the $i$-th entry of the
  column vector~$G$. Let $s_0 = 1$ be the constant series which is the
  solution of the equation $s_0 = Φ(s_0)$.  Let the two families of
  coefficients $(α_{i,j})_{1 ⩽ i,j ⩽ m}$ and $(β_{i,j})_{1 ⩽ i,j ⩽ m}$ be
  defined by $α_{i,j} ≔ μ(0)_{j,i}$ and
  $β_{i,j} ≔ μ(01)_{j,i}= (μ(0)μ(1))_{j,i}$.  Note the inversion of the
  indices $i$ and~$j$.  For each $1 ⩽ i ⩽ m$ the series $s_i$ satisfies
  \begin{displaymath}
    s_i = ∑_{j=1}^m α_{i,j}Φ(s_j) + x∑_{j=1}^m β_{i,j}Φ^2(s_j) 
          + (I_i - ∑_{j=1}^m α_{i,j}I_j)Φ(s_0).
  \end{displaymath}
  These equations come from the fact that the Zeckendorf representation of
  a positive integer ends either with~$0$ or~$01$.  The last term of the
  right hand side deals with the values~$s_i(0)$.  It is equal to zero if
  $Iμ(0) = I$.  Said differently, each non-negative integer is either of
  the form $ϕ(k)$ or $ϕ^2(k)+1$ for some $k ⩾ 0$. Let $n=2m+1$; it is
  chosen so that we force linear dependence and can then get a $Φ$-Mahler
  equation for $f$ as follows.  By
  Lemma~\ref{lem:finite-kernel}, the series $Φ^k(s_i)$ for
  $0 ⩽ k ⩽ n+1 = 2m+2$ and $1 ⩽ i ⩽ m$ are linear combinations of the
  $2m+2$ series $Φ^n(s_j)$ and $Φ^{n+1}(s_j)$ for $0 ⩽ j ⩽ m$. It
  follows that the $2m+3$ series $Φ^k(f)$ for $0 ⩽ k ⩽ 2m+2$ are also
  linear combinations of the $2m+2$ series $Φ^n(s_j)$ and $Φ^{n+1}(s_j)$
  for $0 ⩽ j ⩽ m$ and that the series $Φ^k(f)$ for $0 ⩽ k ⩽ 2m+2$ are not
  linearly independent. Therefore there is a $Z$-Mahler equation $P$ of
  exponent at most $2m+2$ such that $P(x,f)=0$.
\end{proof}

\subsubsection{Dumas' result}\label{sec:Dumas}

As we have noted, Becker and Dumas each showed that a solution of an
isolating $k$-Mahler equation is $k$-regular. In fact Dumas obtained a more
general version, which extends to the Z-numeration as follows.

\begin{theorem} \label{thm:constant-term}
  Let $f(x)$ be the solution of an isolating equation
  \begin{equation}\label{eq:Dumas}
    f(x) = ∑_{i=0}^dA_i(x)Φ^i(f(x)) + g(x)
  \end{equation}
  where $g(x)$ is Z-regular.  Then $f$ is also Z-regular.
\end{theorem}

A weighted automaton is called \emph{normalised} if it has a unique state
with a non-zero final weight and there is no transition with non-zero
weight going out of this state.  The following result is very classical,
see eg~\cite[Proposition~2.14]{Sakarovitch09}.  Recall that two weighted
automata are \emph{equivalent} if they assign the same weight to each word.
\begin{lemma} \label{lem:normalized}
  For each weighted automaton with $n$ states, there is an equivalent
  normalised weighted automaton with $n+1$ states.
\end{lemma}

\begin{lemma} \label{lem:shift-by-one}
  The function $f(x)$ is Z-regular if and only if $xf(x)$ is Z-regular.
\end{lemma}
\begin{proof}
  By Theorem~\ref{thm:Cauchy-product}, the class of Z-regular functions
  is closed under taking products.  Therefore if $f(x)$ is Z-regular, then
  $xf(x)$ is also Z-regular since each polynomial is obviously
  Z-regular. 
  
   Conversely,  suppose that the series $xf(x)$ is
  computed by the weighted automaton~$\mathcal{A}$.  By a variant of 
  Theorem~\ref{thm:regularity-0}, there exists an automaton~$ℬ$ over the
  alphabet $\bar{B} = \{-1, 0, 1\}$ accepting
  \begin{displaymath}
    \{ w ⊟ w' : w,w' ∈ \{0,1\}^* \text{ and } (w)_Z  = (w')_Z + 1 \}.
  \end{displaymath}
  By combining the automaton~$ℬ$ and the weighted automaton~$\mathcal{A}$,
  it is possible to construct a weighted automaton~$\mathcal{C}$ such that
  $\mathcal{C}((n)_Z) = \mathcal{A}((n-1)_Z)$ for each $n ⩾ 1$ and
  $\mathcal{C}(0) = 0$.  This completes the proof of the converse.
\end{proof}

Now we come to the proof of the theorem.
\begin{proof}[Proof of Theorem~\ref{thm:constant-term}]
  Let us suppose that the series~$g= \sum_{n⩾ 0}g_n x^n$ is computed by the
  weighted automaton~$ℬ$.  For each integer~$j ⩾ 0$, iterating
  Lemma~\ref{lem:shift-by-one} $j$ times, there is a weighted
  automaton~$ℬ_j$ such that $ℬ_j((n)_Z) = g_{n-j}$ for each $n ⩾ j$, and
  $ℬ_j((n)_Z) = 0$ otherwise.  By Lemma~\ref{lem:normalized}, it can be
  assumed that each weighted automaton~$ℬ_j$ is normalised.  The weighted
  automaton to compute the solution $f(x)$ of Equation~\eqref{eq:Dumas} is
  obtained by combining the automaton~$\mathcal{A}(P,f_0)$ defined by the
  Mahler equation~\eqref{eq:Z-Mahler-equation} (where there is no $g$),
  with the automata~$ℬ_j$ for $0 ⩽ j ⩽ \tilde{h}$.  The automaton is the
  disjoint union of these automata except that for each integer
  $0 ⩽ j ⩽ \tilde{h}$, the unique final state of~$ℬ_j$ is removed and that
  all transitions ending in that state now end in the states $s_{0,j,q,u}$
  for all possible choices of $q$ and~$u$. We denote the resulting
  automaton by $\mathcal{C}$.

  Equation~\eqref{eq:Dumas} can be rewritten
  \begin{align} \label{eq:Dumas-brut}
    f_n =    ∑_{ϕ^i(k)+ℓ=n} α_{i,ℓ}f_k + g_n .
  \end{align}
  We call $∑_{ϕ^i(k)+ℓ=n} α_{i,ℓ}f_k $ and $g_n$ the first and second
  summands.  We claim that
  $\weight_{\mathcal{C},s_{i,j,q,u}}^{*}(w) = f_k$, where
  $ϕ^i(k) + j = [w]_Z$ and $w=(n)_Z$, as in Proposition~\ref{pro:key-Z}.
  We first consider the case $i = 0$, so that $k + j = [w]_Z$. We use
  Equation~\eqref{eq:Dumas-brut} where $n$ has been replaced by
  $n-j$. Transitions from $\mathcal{A}(P,f_0)$ ending in~$s_{0,j,q,u}$
  contribute the first summand to
  $\weight_{\mathcal{C},s_{0,j,q,u}}^{*}(w)$, and those from~$ℬ_j$
  contribute the second summand to
  $\weight_{\mathcal{C},s_{0,j,q,u}}^{*}(w)$. Thus
  $\weight_{\mathcal{C},s_{0,j,q,u}}^{*}(w)=f_{n-j}$.

  For $i > 0$, the only transitions ending in $s_{i,j,q,u}$ are transitions
  of the form $s_{i-1,ℓ,p,v} \trans{b:1} s_{i,j,q,u}$ which propagate a
  value.  The result follows then by induction on~$i$.
\end{proof}

\subsubsection{A Z-Mahler series which is not Z-regular}
\label{sec:non-regular}

As in the case for $k$-Mahler series \cite[Proposition~1]{Becker-1994}, a
Z-Mahler series is not necessarily Z-regular, as the following example
shows.

\begin{proposition} \label{pro:non-regular}
  The solution $f(x) = ∑_{n ⩾ 0} f_n x^n$ with $f_0 = 1$ of the Z-Mahler
  equation
  \begin{equation}
    (1-x)f(x) = Φ(f(x)) \label{eq:non-isolating-poly}
  \end{equation}
  is not Z-regular.
\end{proposition}

In~\cite[Proposition~1]{Becker-1994}, Becker shows that the solution of the
analogous $k$-Mahler equation $(1-x)f(x) = f(z^k)$ is also not
$k$-regular.  The proof that we provide below is different from the one
given in~\cite{Becker-1994}: it is based on the growth of coefficients.

To prove Proposition~\ref{pro:non-regular}, we define the function $λ: ℕ →
ℕ$, which is the analogue in Zeckendorf numeration of the function $n ↦
⌊n/k⌋$ in base~$k$.  If $(n)_Z = b_k ⋯ b_0$, set
\begin{displaymath}
  λ(n) = λ\left(∑_{i = 0}^k b_iF_i\right) ≔ ∑_{i=1}^k b_iF_{i-1}.
\end{displaymath}
\begin{figure}[htbp]
  \begin{displaymath}
    \begin{array}{r|ccccccccccccccccc}
    n & 0 & 1 & 2 & 3 & 4 & 5 & 6 & 7 & 8 & 9 & 10 & 11 & 12 & 13 & 14 & 15 & 16   \\ \hline
    λ(n) & 0 & 0 & 1 & 2 & 2 & 3 & 3 & 4 & 5 & 5 & 6 & 7 & 7 & 8 & 8 & 9 & 10 \\ 
    \end{array}
  \end{displaymath}
  \caption{The first values of $λ(n)$}
  \label{fig:valuesoflambda}
\end{figure}

The first values of the function~$λ$ are given in
Figure~\ref{fig:valuesoflambda}.  The function~$λ$ is almost an inverse of
the function~$ϕ$ as $λ(ϕ(n)) = n$ for each integer~$n ⩾ 0$, and
\begin{align*}
  ϕ(λ(n)) =
  \begin{cases}
    n   & \text{if the  least significant  digit of  $(n)_Z$ is $0$} \\
    n-1 & \text{otherwise.}
  \end{cases}
\end{align*}
Lemma~\ref{lem:ineq-lambda} follows from these relations between the
functions $λ$ and~$ϕ$, and Lemma~\ref{lem:floor-Cartier}.
\begin{lemma} \label{lem:ineq-lambda}
  There is a positive constant~$c$ such that $λ(n) ⩾ n/φ - c$
  for each integer~$n ⩾ 0$.
\end{lemma}
\begin{proof}[Proof of Proposition~\ref{pro:non-regular}]
  It follows from~\eqref{eq:non-isolating-poly} that each coefficient~$f_i$
  for $i>0$ satisfies 
   \begin{align}\label{eq:def}
    f_i =
    \begin{cases}
      f_{i-1} + f_{λ(i)}  & \text{if $i ∈ \phi(ℕ)$} \\
      f_{i-1}            & \text{otherwise.}
    \end{cases}
  \end{align}
  Summing up these equations for $i = 0, … , n$ yields
  \begin{displaymath}
    \sum_{i=0}^nf_i = \sum_{i=0}^{n-1}f_i + \sum_{i=0}^{λ(n)} f_i
  \end{displaymath}
  and thus
  \begin{equation}\label{eq:referee}
    f_n = \sum_{i=0}^{λ(n)} f_i .
  \end{equation}
  We now prove by induction on~$k$ that for each integer~$k$, there exists
  a positive constant~$C_k$ such that $f_n ⩾ C_kn^k$ holds for each integer
  $n ⩾ 0$.  It follows easily from~\eqref{eq:def} that each
  coefficient~$f_n$ satisfies $f_n ⩾ 0$ and $f_n ⩾ f_{n-1}$ for $n ⩾ 1$.
  Since $f_0 = 1$, each coefficient satisfies $f_n ⩾ 1$ and the result is
  proved for $k = 0$ with $C_0 = 1$.  Suppose that the statement is true 
   for some $k ⩾ 0$. From~\eqref{eq:referee}, we get
  \begin{displaymath}
    f_n ⩾ C_k∑_{i=0}^{λ(n)} i^k .
  \end{displaymath}
  Using the relation stated in Lemma~\ref{lem:ineq-lambda}, and using
  Faulhaber's formula, we have
  \begin{displaymath}
    f_n ⩾ C_k∑_{i=0}^{λ(n)} i^k ⩾  C_k∑_{i=0}^{n/φ - c} i^k
        ⩾ \frac{C_k}{k+1}\sum_{r=0}^{k} \binom{k+1}{r} B_r^+(n/φ - c)^{k+1-r},
  \end{displaymath}
  and from this one obtains a positive constant~$C_{k+1}$ such that
  $f_n ⩾ C_{k+1}n^{k+1}$ holds for each $n ⩾ 0$.  This proves the statement
  for $k+1$.  By Lemma~\ref{lem:asymptotic-bound}, $f$ is not Z-regular.
\end{proof}

The same technique of Lemma~\ref{lem:ineq-lambda} can be applied to show
that solutions of other Z-Mahler equations are also not Z-regular, with the
same arguments For instance, the non-zero solution of the equation
$(1-x)f(x) = \frac{1}{2} \left(Φ(f(x))+Φ^2(f(x))\right)$, as well as the
solution of the equation is $(1-x^2)f(x) = Φ(f(x))$ are also not Z-regular.
It also follows from the proof of Proposition~\ref{pro:non-regular} that
the solution of the Z-Mahler equation $(1-α x)f(x) = Φ(f(x))$ for $α ⩾ 1$
is also not Z-regular because the coefficients obviously grow faster than
the ones of the solution of $(1-x)f(x) = Φ(f(x))$. The same approach does
not seem to work for the similar Z-Mahler equation $(1+x^2)f(x) = Φ(f(x))$
although it is reasonable to assume that the solution of this equation is
also not Z-regular.  Note that the solution of $(1+x)f(x) = Φ(f(x))$ turns
out to be Z-regular because it is the polynomial $1-x$.

\section*{Conclusion}

In conclusion, we mention a few open problems.  Allouche and Shallit
proved in~\cite[Theorem~2.11]{Allouche-Shallit-1992} that a geometric series $f(x)
= ∑_{n ⩾ 0}α^n x^n$ is $k$-regular if and only if $α$ is either zero or a
root of unity.  We do not have a similar result for Z-regular series. Using
the result by Allouche and Shallit, in \cite{Bell-Chyzak-Coons-Dumas}, the authors characterised
$k$-regular series in terms of the $k$-Mahler equations they satisfy.  We
do not know if there is a similar characterisation for Z-Mahler
equations.

In~\cite[Proposition~7.8]{AdamczewskiBell13}, Adamczewski and Bell gave a series
which is $k$-regular but which is not the solution of an isolating
$k$-Mahler equation ($k$-Becker in their terminology).  We do not have such
an example, of a series which is Z-regular but which is not a solution of
an isolating Z-Mahler equation.

Other questions include whether one can extend existing Cobham type
results. For example, in \cite{AdamczewskiBell13} and \cite{SS-2019}, it was
shown that a series which is both $k$- and $\ell$-Mahler over a field of
characteristic zero, with $k$ and $\ell$ multiplicatively independent, must be
rational.  Which series are both $k$- and Z-Mahler?

\section*{Acknowledgments}
The authors are very grateful to the referee for a heroic and highly useful
report, which has greatly improved this paper.

\bibliographystyle{amsalpha}
\bibliography{christol}

@article {Peltomaki-Salo-2022,
    AUTHOR = {Peltom\"{a}ki, Jarkko and Salo, Ville},
     TITLE = {Automatic winning shifts},
   JOURNAL = {Inform. and Comput.},
  FJOURNAL = {Information and Computation},
    VOLUME = {285},
      YEAR = {2022},
    NUMBER = {part B},
     PAGES = {Paper No. 104883, 21},
      ISSN = {0890-5401,1090-2651},
   MRCLASS = {68Q70 (37B15 91A46)},
  MRNUMBER = {4434428},
       DOI = {10.1016/j.ic.2022.104883},
       URL = {https://doi-org.ezproxy.library.qmul.ac.uk/10.1016/j.ic.2022.104883},
}

@article {Maes-Rigo-2002,
    AUTHOR = {Rigo, Michel and Maes, Arnaud},
     TITLE = {More on generalized automatic sequences},
   JOURNAL = {J. Autom. Lang. Comb.},
  FJOURNAL = {Journal of Automata, Languages and Combinatorics},
    VOLUME = {7},
      YEAR = {2002},
    NUMBER = {3},
     PAGES = {351--376},
      ISSN = {1430-189X,2567-3785},
   MRCLASS = {68Q45 (11B85)},
  MRNUMBER = {1957696},
MRREVIEWER = {Val\'erie\ Berth\'e},
}

@article{CCS,
	author = {Charlier, \'{E}milie and Cisternino, C\'{e}lia and Stipulanti, Manon},
	doi = {10.1016/j.ejc.2021.103475},
	fjournal = {European Journal of Combinatorics},
	issn = {0195-6698},
	journal = {European J. Combin.},
	mrclass = {11B85 (68Q70)},
	mrnumber = {4338412},
	mrreviewer = {Yao-Qiang Li},
	pages = {Paper No. 103475, 34},
	title = {Regular sequences and synchronized sequences in abstract numeration systems},
	url = {https://doi-org.ezproxy.library.qmul.ac.uk/10.1016/j.ejc.2021.103475},
	volume = {101},
	year = {2022}}

@article{SS-2019,
	author = {Sch\"{a}fke, Reinhard and Singer, Michael},
	doi = {10.4171/JEMS/891},
	fjournal = {Journal of the European Mathematical Society (JEMS)},
	issn = {1435-9855},
	journal = {J. Eur. Math. Soc. (JEMS)},
	mrclass = {39A05 (34A30 34M03 39A13 39A45)},
	mrnumber = {3985611},
	mrreviewer = {P. W. Eloe},
	number = {9},
	pages = {2751--2792},
	title = {Consistent systems of linear differential and difference equations},
	url = {https://doi-org.ezproxy.library.qmul.ac.uk/10.4171/JEMS/891},
	volume = {21},
	year = {2019}}

@article{AdamczewskiBell13,
	author = {Adamczewski, Boris and Bell, Jason P.},
	fjournal = {Annali della Scuola Normale Superiore di Pisa. Classe di Scienze. Serie V},
	issn = {0391-173X},
	journal = {Ann. Sc. Norm. Super. Pisa Cl. Sci. (5)},
	mrclass = {11J81 (11B85)},
	mrnumber = {3752528},
	mrreviewer = {Michael Welter},
	number = {4},
	pages = {1301--1355},
	title = {A problem about {M}ahler functions},
	volume = {17},
	year = {2017}}

@article{Bell-Chyzak-Coons-Dumas,
	author = {Bell, Jason P. and Chyzak, Fr\'{e}d\'{e}ric and Coons, Michael and Dumas, Philippe},
	doi = {10.1090/tran/7762},
	fjournal = {Transactions of the American Mathematical Society},
	issn = {0002-9947},
	journal = {Trans. Amer. Math. Soc.},
	mrclass = {30B10 (11B85 68R15)},
	mrnumber = {3988615},
	mrreviewer = {L. R. Sons},
	number = {5},
	pages = {3405--3423},
	title = {Becker's conjecture on {M}ahler functions},
	url = {https://doi.org/10.1090/tran/7762},
	volume = {372},
	year = {2019}}

@article{Adamczewski,
	author = {Adamczewski, Boris},
	fjournal = {Documenta Mathematica},
	issn = {1431-0635},
	journal = {Doc. Math.},
	mrclass = {11J81 (11-03)},
	mrnumber = {4604998},
	number = {Extra vol.: Mahler Selecta},
	pages = {95--122},
	title = {Mahler's method},
	year = {2019}}

@incollection{Coons-Spiegelhofer,
	author = {Coons, Michael and Spiegelhofer, Lukas},
	booktitle = {Sequences, groups, and number theory},
	doi = {10.1007/978-3-319-69152-7\_2},
	mrclass = {11B83 (11B85)},
	mrnumber = {3799924},
	mrreviewer = {Robert F. Tichy},
	pages = {37--87},
	publisher = {Birkh\"{a}user/Springer, Cham},
	series = {Trends Math.},
	title = {Number theoretic aspects of regular sequences},
	url = {https://doi.org/10.1007/978-3-319-69152-7_2},
	year = {2018}}

@article{Mahler-1982,
	author = {Mahler, Kurt},
	doi = {10.1016/0022-314X(82)90044-0},
	fjournal = {Journal of Number Theory},
	issn = {0022-314X},
	journal = {J. Number Theory},
	mrclass = {01A70 (10-03)},
	mrnumber = {655723},
	mrreviewer = {W. W. Adams},
	number = {2},
	pages = {121--155},
	title = {Fifty years as a mathematician},
	url = {https://doi.org/10.1016/0022-314X(82)90044-0},
	volume = {14},
	year = {1982}}

@book{Nishioka,
	author = {Nishioka, Kumiko},
	doi = {10.1007/BFb0093672},
	isbn = {3-540-61472-9},
	mrclass = {11J81},
	mrnumber = {1439966},
	mrreviewer = {A. J. van der Poorten},
	pages = {viii+185},
	publisher = {Springer-Verlag, Berlin},
	series = {Lecture Notes in Mathematics},
	title = {Mahler functions and transcendence},
	url = {https://doi.org/10.1007/BFb0093672},
	volume = {1631},
	year = {1996}}

@book{Dumas-1993,
	author = {Dumas, Philippe},
	isbn = {2-7261-0817-2},
	mrclass = {11B85 (05A16)},
	mrnumber = {1346304},
	mrreviewer = {Renzo Sprugnoli},
	note = {Th\`ese, Universit\'{e} de Bordeaux I, Talence, 1993},
	pages = {241},
	publisher = {Institut National de Recherche en Informatique et en Automatique (INRIA), Rocquencourt},
	title = {R\'{e}currences mahl\'{e}riennes, suites automatiques, \'{e}tudes asymptotiques},
	year = {1993}}

@article{Christol-1980,
	author = {Christol, Gilles and Kamae, Teturo and Mend\`es France, Michel and Rauzy, G\'{e}rard},
	fjournal = {Bulletin de la Soci\'{e}t\'{e} Math\'{e}matique de France},
	issn = {0037-9484},
	journal = {Bull. Soc. Math. France},
	mrclass = {10K99},
	mrnumber = {614317},
	mrreviewer = {Maurice Mignotte},
	number = {4},
	pages = {401--419},
	title = {Suites alg\'{e}briques, automates et substitutions},
	url = {http://www.numdam.org/item?id=BSMF_1980__108__401_0},
	volume = {108},
	year = {1980}}

@article{Christol-1979,
	author = {Christol, Gilles},
	doi = {10.1016/0304-3975(79)90011-2},
	fjournal = {Theoretical Computer Science},
	issn = {0304-3975},
	journal = {Theoret. Comput. Sci.},
	mrclass = {68D05 (05A99)},
	mrnumber = {535129},
	number = {1},
	pages = {141--145},
	title = {Ensembles presque periodiques {$k$}-reconnaissables},
	url = {https://doi-org.ezproxy.library.qmul.ac.uk/10.1016/0304-3975(79)90011-2},
	volume = {9},
	year = {1979}}

@book{Meyer-1972,
	author = {Meyer, Yves},
	mrclass = {12A15 (10F45 42A44 43A46)},
	mrnumber = {485769},
	mrreviewer = {Henri Joris},
	pages = {x+274},
	publisher = {North-Holland Publishing Co., Amsterdam-London; American Elsevier Publishing Co., Inc., New York},
	series = {North-Holland Mathematical Library, Vol. 2},
	title = {Algebraic numbers and harmonic analysis},
	year = {1972}}

@article{Lee-Solomyak,
	author = {Lee, Jeong-Yup and Solomyak, Boris},
	doi = {10.3934/dcds.2012.32.935},
	fjournal = {Discrete and Continuous Dynamical Systems. Series A},
	issn = {1078-0947},
	journal = {Discrete Contin. Dyn. Syst.},
	mrclass = {37B50 (52C23)},
	mrnumber = {2851885},
	mrreviewer = {Thomas Ward},
	number = {3},
	pages = {935--959},
	title = {Pisot family self-affine tilings, discrete spectrum, and the {M}eyer property},
	url = {https://doi-org.ezproxy.library.qmul.ac.uk/10.3934/dcds.2012.32.935},
	volume = {32},
	year = {2012}}

@incollection{Allouche-1992,
	author = {Allouche, Jean-Paul},
	booktitle = {L{ATIN} '92 ({S}\~{a}o {P}aulo, 1992)},
	doi = {10.1007/BFb0023813},
	mrclass = {11B85 (68Q70)},
	mrnumber = {1253343},
	mrreviewer = {Christian Mauduit},
	pages = {15--23},
	publisher = {Springer, Berlin},
	series = {Lecture Notes in Comput. Sci.},
	title = {{$q$}-regular sequences and other generalizations of {$q$}-automatic sequences},
	url = {https://doi-org.ezproxy.library.qmul.ac.uk/10.1007/BFb0023813},
	volume = {583},
	year = {1992}}

@article{Shallit-1988,
	author = {Shallit, Jeffrey},
	doi = {10.1016/0304-3975(88)90103-X},
	fjournal = {Theoretical Computer Science},
	issn = {0304-3975},
	journal = {Theoret. Comput. Sci.},
	mrclass = {68Q45},
	mrnumber = {974766},
	mrreviewer = {Klaus W. Wagner},
	number = {1},
	pages = {1--16},
	title = {A generalization of automatic sequences},
	url = {https://doi-org.ezproxy.library.qmul.ac.uk/10.1016/0304-3975(88)90103-X},
	volume = {61},
	year = {1988}}

@book{Lothaire,
	author = {Lothaire, M.},
	doi = {10.1017/CBO9781107326019},
	isbn = {0-521-81220-8},
	mrclass = {68R15 (05A99 05E10 20M35)},
	mrnumber = {1905123},
	mrreviewer = {Patrice\ S\'{e}\'{e}bold},
	pages = {xiv+504},
	publisher = {Cambridge University Press, Cambridge},
	series = {Encyclopedia of Mathematics and its Applications},
	title = {Algebraic combinatorics on words},
	url = {https://doi.org/10.1017/CBO9781107326019},
	volume = {90},
	year = {2002}}

@article{Rigo-2000,
	author = {Rigo, Michel},
	doi = {10.1016/S0304-3975(00)00163-8},
	fjournal = {Theoretical Computer Science},
	issn = {0304-3975,1879-2294},
	journal = {Theoret. Comput. Sci.},
	mrclass = {68Q45 (11B85)},
	mrnumber = {1774400},
	mrreviewer = {V\'{e}ronique\ Bruy\`ere},
	number = {1-2},
	pages = {271--281},
	title = {Generalization of automatic sequences for numeration systems on a regular language},
	url = {https://doi.org/10.1016/S0304-3975(00)00163-8},
	volume = {244},
	year = {2000}}

@article{Allouche-Shallit-1992,
	author = {Allouche, Jean-Paul and Shallit, Jeffrey},
	doi = {10.1016/0304-3975(92)90001-V},
	fjournal = {Theoretical Computer Science},
	issn = {0304-3975,1879-2294},
	journal = {Theoret. Comput. Sci.},
	mrclass = {11B85 (05A99 68Q70)},
	mrnumber = {1166363},
	mrreviewer = {Jau-Shyong\ Shiue},
	number = {2},
	pages = {163--197},
	title = {The ring of {$k$}-regular sequences},
	url = {https://doi.org/10.1016/0304-3975(92)90001-V},
	volume = {98},
	year = {1992}}

@article{Becker-1994,
	author = {Becker, Paul-Georg},
	doi = {10.1006/jnth.1994.1093},
	fjournal = {Journal of Number Theory},
	issn = {0022-314X,1096-1658},
	journal = {J. Number Theory},
	mrclass = {11B85 (11J91)},
	mrnumber = {1307967},
	mrreviewer = {John\ H.\ Loxton},
	number = {3},
	pages = {269--286},
	title = {{$k$}-regular power series and {M}ahler-type functional equations},
	url = {https://doi.org/10.1006/jnth.1994.1093},
	volume = {49},
	year = {1994}}

@book{Allouche-Shallit,
	author = {Allouche, Jean-Paul and Shallit, Jeffrey},
	doi = {10.1017/CBO9780511546563},
	isbn = {0-521-82332-3},
	mrclass = {11B85 (11Z05 37A45 37B10 68Q45 68R15 94A45)},
	mrnumber = {1997038},
	mrreviewer = {Val\'{e}rie Berth\'{e}},
	note = {Theory, applications, generalizations},
	pages = {xvi+571},
	publisher = {Cambridge University Press, Cambridge},
	title = {Automatic sequences},
	url = {https://doi.org/10.1017/CBO9780511546563},
	year = {2003}}

@article{Frougny-Solomyak-1992,
	author = {Frougny, Christiane and Solomyak, Boris},
	doi = {10.1017/S0143385700007057},
	fjournal = {Ergodic Theory and Dynamical Systems},
	issn = {0143-3857},
	journal = {Ergodic Theory Dynam. Systems},
	mrclass = {11K55 (11A67)},
	mrnumber = {1200339},
	mrreviewer = {Robert F. Tichy},
	number = {4},
	pages = {713--723},
	title = {Finite beta-expansions},
	url = {https://doi.org/10.1017/S0143385700007057},
	volume = {12},
	year = {1992}}

@article{Frougny-1992,
	author = {Frougny, Christiane},
	doi = {10.1007/BF01368783},
	fjournal = {Mathematical Systems Theory. An International Journal on Mathematical Computing Theory},
	issn = {0025-5661},
	journal = {Math. Systems Theory},
	mrclass = {68Q70 (11A67 11Y16)},
	mrnumber = {1139094},
	mrreviewer = {Krassimir Atanassov},
	number = {1},
	pages = {37--60},
	title = {Representations of numbers and finite automata},
	url = {https://doi.org/10.1007/BF01368783},
	volume = {25},
	year = {1992}}

@article{Furstenberg-1967,
	author = {Furstenberg, Harry},
	doi = {10.1016/0021-8693(67)90061-0},
	fjournal = {Journal of Algebra},
	issn = {0021-8693},
	journal = {J. Algebra},
	mrclass = {12.78},
	mrnumber = {215820},
	mrreviewer = {L. J. Ratliff, Jr.},
	pages = {271--277},
	title = {Algebraic functions over finite fields},
	url = {https://doi-org.ezproxy.library.qmul.ac.uk/10.1016/0021-8693(67)90061-0},
	volume = {7},
	year = {1967}}

@book{Sakarovitch09,
	author = {Jacques Sakarovitch},
	publisher = {Cambridge University Press},
	title = {Elements of Automata Theory},
	year = 2009}

\end{document}